\documentclass[draft]{amsart}
\usepackage{amssymb,graphicx}
\usepackage{multirow}
\usepackage{comment}

\newcommand{\df}{\dfrac}
\newcommand{\tf}{\tfrac}

\renewcommand{\i}{\infty}

\newcommand{\beqs}{\begin{equation*}}
\newcommand{\eeqs}{\end{equation*}}
\numberwithin{equation}{section}
 \theoremstyle{plain}
\newtheorem{theorem}{Theorem}[section]

\theoremstyle{remark}

\begin{document}

\makeatletter
\def\imod#1{\allowbreak\mkern10mu({\operator@font mod}\,\,#1)}
\makeatother

\author{Alexander Berkovich}
   \address{Department of Mathematics, University of Florida, 358 Little Hall, Gainesville FL 32611, USA}
   \email{alexb@ufl.edu}

\author{Frank Patane}
   \address{Department of Mathematics, University of Florida, 358 Little Hall, Gainesville FL 32611, USA}
   \email{frankpatane@ufl.edu}

\title[\scalebox{.9}{Binary quadratic forms and the Fourier coefficients of certain weight 1 eta-quotients}]{Binary Quadratic forms and the Fourier coefficients of certain weight 1 eta-quotients}
     
\begin{abstract} 
We state and prove an identity which represents the most general \textit{eta}-products of weight 1 by binary quadratic forms. We discuss the utility of binary quadratic forms in finding a multiplicative completion for certain \textit{eta}-quotients. We then derive explicit formulas for the Fourier coefficients of certain \textit{eta}-quotients of weight 1 and level $47, 71, 135, 648, 1024,$ and $1872$.
\end{abstract}



\date{\today}
   
\maketitle

   \section{Introduction and Notation}
\label{intro}

Throughout the paper we assume $q$ is a complex number with $|q|<1$. We use the standard notations
   \begin{equation}
   \label{poch}
   (a;q)_\i:= \prod_{n=0}^{\infty}(1-aq^n)
   \end{equation}
	and
   \begin{equation}
   \label{E}
   E(q):= (q;q)_\i.
   \end{equation}
   Next, we recall the Ramanujan theta function
   
   \begin{equation}
   \label{fdef}
   f(a,b):= \sum_{n=-\infty}^{\infty}a^{\tf{n(n+1)}{2}}b^{\tf{n(n-1)}{2}}, \hspace{1cm} |ab|<1.
   \end{equation}
	
  \noindent 
   The function $ f(a,b) $ satisfies the Jacobi triple product identity \cite[Entry 19]{ramnote}
\begin{equation}
\label{jtp}
f(a,b) = (-a;ab)_{\i}(-b;ab)_{\i}(ab;ab)_{\i},
\end{equation}
   along with
	\begin{equation}
\label{fshift}
f(a,b) =a^{\tf{n(n+1)}{2}}b^{\tf{n(n-1)}{2}}f(a(ab)^n,b(ab)^{-n}),
\end{equation}

\noindent
where $n \in \mathbb{Z} $ \cite[Entry 18]{ramnote}.
   One may use \eqref{jtp} to derive the following special cases:
   
   \begin{equation}
   \label{pent}
   E(q) =f(-q,-q^2)= \sum_{n=-\infty}^\infty (-1)^n q^{\tf{n(3n-1)}{2}},
   \end{equation}
   \begin{equation}
   \label{phi}
   \phi(q):=f(q,q)= \sum_{n=-\infty}^{\infty}q^{n^2} = \df{E^5(q^2)}{E^2(q^4)E^2(q)},
   \end{equation}
   \begin{equation}
   \label{psi}
   \psi(q):=f(q,q^3)= \sum_{n=-\infty}^{\infty}q^{2n^2-n} = \df{E^2(q^2)}{E(q)},
   \end{equation}
   \begin{equation}
   \label{f12}
   f(q,q^2)= \df{E^2(q^3)E(q^2)}{E(q^6)E(q)},
   \end{equation}
   \begin{equation}
   \label{f15}
   f(q,q^5)= \df{E(q^{12})E^2(q^2)E(q^3)}{E(q^6)E(q^4)E(q)}.
   \end{equation}
	
   Note that \eqref{pent} is the famous Euler pentagonal number theorem.
   Splitting \eqref{pent} according to the parity of the index of summation, we find
   
   \begin{equation}
   \label{pentcor}
   E(q)= \sum_{n=-\infty}^{\infty}q^{6n^2 -n}-q\sum_{n=-\infty}^{\infty}q^{6n^2 +5n} =f(q^5,q^7) -qf(q,q^{11}).
   \end{equation}
   We will employ the similarly derived relations
   \begin{equation}
   \label{mod31}
   \phi(q)= \phi(q^9) +2qf(q^3,q^{15}),
   \end{equation}
   \begin{equation}
   \label{phieven}
   \phi(q) = \phi(q^4) +2q\psi(q^8),
   \end{equation}
\begin{equation}
	\label{psipsiaux}
\psi(-q) = f(q^6,q^{10}) -qf(q^2,q^{14}),
\end{equation}
and
   \begin{equation}
   \label{mod32}
   \psi(q)=f(q^3,q^6) +q\psi(q^9).
   \end{equation}
   We comment that the relation
	\[
	E(-q) = \df{E^3(q^2)}{E(q^4)E(q)}
	\]
	can be used to deduce
	\[
	\phi(-q)=\df{E^2(q)}{E(q^2)},
	\]
	and
	\[
	\psi(-q)=\df{E(q)E(q^4)}{E(q^2)}.
	\]
	
      We now recall some notation from elementary number theory. Given an integer $n$, we use the term $\operatorname{ord}_{p}(n)$ to represent the unique integer with the properties $p^{\operatorname{ord}_{p}(n)} \mid n$ and $p^{1+\operatorname{ord}_{p}(n)} \nmid n$.\\
      \noindent
   For an odd prime $p$, Legendre's symbol is defined by 
   \[
   \left(\df{n}{p}\right) = \left\{ \begin{array}{ll}
       1 &  \mbox{ if $n$ is a quadratic residue modulo $p$ and $p \nmid n$}, \\
       -1 &  \mbox{ if $n$ is a quadratic nonresidue modulo $p$ and $p \nmid n$}, \\
       0 &  \mbox{ if $p \mid n$ }.
     \end{array}
     \right.
   \]
   \noindent
   Kronecker's symbol $\left(\df{n}{m}\right)$ is defined by
   \[
   \left(\df{n}{m}\right) = \left\{ \begin{array}{ll}
        1&  \mbox{ if $m=1$ }, \\
       0&  \mbox{ if $m$ is a prime dividing $n$},\\
       \mbox{Legendre's symbol} &  \mbox{ if $m$ is an odd prime},
     \end{array}
     \right.
   \]
   \[
   \left(\df{n}{2}\right) = \left\{ \begin{array}{ll}
        0&  \mbox{ if $n$ is even }, \\
       1&  \mbox{  $n \equiv \pm1 \pmod{8}$},\\
       -1&  \mbox{  $n \equiv \pm3 \pmod{8}$},
     \end{array}
     \right.
   \]
   and in general $\left(\df{n}{m}\right) = \prod_{i=1}^s \left(\df{n}{p_i}\right)$, where $m=\prod_{i=1}^s p_i$ is a prime factorization of $m$.

   	We recall some well-known facts from the theory of binary quadratic forms. For $a,b,c \in \mathbb{Z}$, we call $(a,b,c):= ax^2 +bxy +cy^2$ a binary quadratic form of discriminant $d:=b^2-4ac$.  We are only interested in the case $d<0$ and $a>0$, and the forms in this case are called positive definite. From now on we will use the abbreviated term ``form" to refer to a positive definite binary quadratic form. If we set $g(x,y) = ax^2 +bxy +cy^2$, then the binary quadratic form $h(x,y):=g(\alpha x +\beta y, \gamma x + \delta y)$ is called equivalent to $g(x,y)$, when $\left( \begin{array}{cc}
\alpha & \beta\\
\gamma & \delta \end{array} \right)$ is in SL$(2,\mathbb{Z})$. We write $g(x,y) \sim h(x,y)$ and note that $\sim$ is an equvalence relation on the set of binary quadratic forms of discriminant $d$. We define $H(d)$ to be the set of all binary quadratic forms of discriminant $d$ modulo the equivalence relation $\sim$. It is well-know that $H(d)$ has a group structure, and is called the class group of dicriminant $d$ \cite{buell}. The conductor of $d$ is the largest integer $f$ such that $\tf{d}{f^2}$ is also a discriminant of binary quadratic forms.\\

  The form $(a,b,c)$ has the associated theta series
   \begin{equation}
   \label{defB}
   B(a,b,c,q):=\sum_{x,y}q^{ax^2 +bxy +cy^2}= \sum_{n \geq 0}(a,b,c,n)q^n.
   \end{equation}
	Note that $(a,b,c,n)=0$ whenever $n \notin \mathbb{Z}_{\geq 0}$.\\
	
	Let $w(d) = 6, 4,$ or $ 2$ if $d= -3, -4$ or $d<-4$, respectively. The prime $p$ is represented by a form of discriminant $d$ if and only if $p\nmid f$ and $\left(\tf{d}{p}\right)=0,1$. Let us take $(a,b,c)$ to be a form of discriminant $d$ which represents $p$. Since $(a,b,c)$ and $(a,-b,c)$ represent the same integers, we see $(a,-b,c)$ must also represent $p$. Apart from $(a,b,c)$ and $(a,-b,c)$, no other class of forms of discriminant $d$ will represent $p$. If $p\nmid f$ and $\left(\tf{d}{p}\right)=0$, then $(a,b,c,p) = w(d)$.  If $\left(\tf{d}{p}\right)=1$, then $(a,b,c,p)=w(d)$ or $2w(d)$ when $(a,b,c) \nsim (a,-b,c)$ or $(a,b,c)\sim (a,-b,c)$, respectively \cite{dickson}.\\

   We will need a method to determine which primes $p$ are represpresented by a given form. Besides using basic congruence conditions from genus theory, see \cite{dickson}, we follow \cite{cox} and \cite{voight} by employing the Weber class polynomial. For a form $(a,b,c)$ with discriminant $d$, we let $\tau = \tf{-b+\sqrt{d}}{a}$, and recall that the Weber class polynomial, $W_d$, is defined as the minimal polynomial of $f(\tau)$, where $f$ is a particular normalized Weber function generating the same class field as Klein's modular function $j(\tau)$. (See \cite{cox}, \cite{voight} for details.)\\	
	
We will give examples where the factorization pattern of $W_d \imod{p}$ suffices to determine if $(a,b,c)$ represents $p$. When the factorization pattern of $W_d \imod{p}$ does not suffice, we use rem$(z^p,W_d(z))\imod{p}$.  Here and throughout the manuscript, rem$(z^p,W_d(z))\imod{p}$ denotes the remainder of $z^p$ upon division by $W_d(z)$ modulo $p$. If $\left(\tf{d}{p}\right)=1$, an analysis of rem$(z^p,W_d(z))\imod{p}$ always enables us to determine which form of discriminant $d$ represents $p$.\\

   An \textit{eta}-quotient is a finite product 
		\[
	H(q):=q^j\prod_{i=1}^{M}E^{r_{i}}(q^{s_{i}}),
	\]
	where $s_1,\ldots, s_M$ are positive integers, $r_1,\ldots, r_M$ are non-zero integers, and\\ $j = \sum_{i=1}^{M}\tf{r_{i}s_{i}}{24} \in \mathbb{Z}$. We call $H$ an \textit{eta}-quotient since $q^{\tf{1}{24}}E(q)=\eta(q)$, where $\eta(q)$ is the Dedekind eta function. The weight of $H$ is defined to be $\sum_{i=1}^{M}\tf{r_{i}}{2}.$ The level $N$ of $H$ is defined to be the smallest multiple of lcm$(s_1,\ldots, s_M)$ such that $\sum_{i=1}^{M}\tf{r_{i}N}{s_{i}} \equiv 0 \imod{24}.$  
   $H$ has a Fourier expansion, and we use the notation $[q^n]H(q)$ to denote the coefficient of $q^n$ in the expansion of $H$. We call a series $\sum_{n>0}a(n)q^n$ multiplicative when $a(n)$ is multiplicative. It is a continuing area of research to determine the Fourier coefficients of certain \textit{eta}-quotients. (See \cite{berk1}, \cite{berk2}, \cite{ccl}, \cite{ernest}, \cite{kohler}, and \cite{williams}.) \\
  
   Let $Q$ be a form of discriminant $d=d_0\cdot f^2$ with conductor $f$. Let $\vartheta_Q=\sum_{n \geq 0} h(n)q^n$ be the associated theta series for $Q$. In \cite{hecke} Hecke defines the operator $T_p$ by
   \begin{equation}
   \label{heckedef}
   T_p \left( \sum_{n=0}^\i h(n)q^n \right)= \sum_{n=0}^\i \left[h(pn)+ \left(\tf{d}{p}\right)h(n/p)\right]q^n.
   \end{equation}

	Let $S$ be the set of all linear combinations of theta series associated to the forms of discriminant $d$. It is well-known that for $s\in S$, we have $T_p (s) \in S $ \cite{hecke}. Moreover, Hecke gives explicit information regarding which theta series are involved in the linear combination $T_p(s)$ \cite[p.~794; (7)]{hecke}. In the case that $p$ is represented by a form of discriminant $d$, we define $Q_0$ to be such a form. It can be shown that
	  \begin{equation}
      T_p(\vartheta_Q) = \left\{ \begin{array}{ll}
       \vartheta_{Q\cdot Q_0} + \vartheta_{Q\cdot Q^{-1}_0} &  \left(\tf{d}{p}\right)=1, \\
       \vartheta_{Q\cdot Q_0} & p \mid d, p \nmid f , \\
			 F(q^{p}) & p \mid f, \\
       0 &   \left(\tf{d}{p}\right)=-1, \\
     \end{array}
     \right.
\end{equation}
	where the binary operation $\cdot$ is Gaussian composition of forms, and $F(q)=\vartheta_{Q_1}$ where $Q_1$ is a particular form of discriminant $\tf{d}{p^2}$.\\
	
    When $T_p (s) =s\cdot \lambda_p$ for some constant $\lambda_p$, we call $s$ an eigenform of $T_p$ with eigenvalue $\lambda_p$. If $\sum h(n) q^n$ is an eigenform for all Hecke operators, then \eqref{heckedef} implies
   \begin{equation}
   \label{heckedef2}
   \lambda_p h(n)= h(pn)+ \left(\tf{d}{p}\right)h(n/p)
   \end{equation}
   for any positive integer $n$ and any prime $p$. Taking $n=1$ yields $\lambda_p\cdot h(1) = h(p)$, and we will always take $h$ appropriately normalized so that $h(1)=1$.	Hence the eigenvalues of all $T_p$ are exactly the evaluations of $h$ at the primes.
   Equation \eqref{heckedef2} shows $h$ is multiplicative. Moreover, taking $n=p^k$ in \eqref{heckedef2} yields
   \begin{equation}
   \label{heckedef4}
   h(p^{k+1}) = h(p)h(p^k)- \left(\tf{d}{p}\right)h(p^{k-1}).\\
   \end{equation}

  Given $s, s^{\prime}\in S$ we say that $s+ s^{\prime}$ is a completion of $s$ if $s+s^{\prime}$ is a an eigenform for all Hecke operators. We say that $s$ and $s^{\prime}$ are congruentially disjoint if there exists a modulus $M$, such that for any $n,m$ with $[q^n]s \neq 0$, $ [q^m]s^{\prime} \neq 0$, we have $n \not\equiv m \imod{M}$.\\
	
	In later sections, we find the Fourier coefficients of certain \textit{eta}-quotients $H$, by finding a completion $s^{\prime}+H$ with the additional property that $s^{\prime}$ and $H$ are congruentially disjoint. The ability to extract the coefficients of $H$ by employing congruences, contributes to the simplicity of our results.	\\
	
	To find a completion, we use certain multiplicative combinations discussed in \cite{sunwil}. Let $A^{n}(q)$ be the associated theta series of the binary quadratic form $A^{n}$ ($A$ composed with itself $n$ times). We employ a result of Sun and Williams \cite{sunwil} to find the following multiplicative linear combinations:

	\begin{table}[htb]  
	\caption {} 	\label{sunwilmc} 
     \begin{center} 
       \begin{tabular}{|l|c|}              \hline 
			$C_5 \cong \langle A \rangle $ & $\tf{1}{2}\left(I(q)-\mu A(q)-\lambda A^2(q) \right)$ \\\cline{2-2}
				 & $\tf{1}{2}\left(I(q)-\lambda A(q)-\mu A^2(q) \right)$ \\[.1cm]  \hline
				$C_7 \cong \langle A \rangle $ & $\tf{1}{2}\left(I(q)+\df{}{}\alpha A(q)+\beta A^2(q) +\gamma A^3(q) \right)$ \\\cline{2-2}
				 & $\tf{1}{2}\left(I(q)+\beta A(q)+\df{}{}\gamma A^2(q) +\alpha A^3(q) \right)$  \\\cline{2-2}
				 & $\tf{1}{2}\left(I(q)+\gamma \df{}{}A(q)+\alpha A^2(q) +\beta A^3(q) \right)$ \\[.05cm]  \hline
         $C_6 \cong \langle A \rangle $ & $\tf{1}{2}\left(I(q)-\df{}{} A^2(q) +A(q) -A^3(q)\right)$ \\[.1cm]  \hline
				$C_8 \cong \langle A \rangle $ & $\tf{1}{2}\left(I(q)- A^4(q) + \df{}{}\sqrt{2}A(q) -\sqrt{2}A^3(q)\right)$ \\[.1cm]  \hline
				$C_4 \times C_4 \cong \langle A,B \rangle $ & $\tf{1}{2} \left(I(q)+\df{}{} A^2(q) -B^2(q) -A^2B^2(q) +2A(q) -2AB^2(q)\right)$ \\  \hline
       \end{tabular}
			 \begin{flushright}
			 ,
			 \end{flushright}
			 \end{center}
   \end{table}
	\noindent
	where $I(q):=A^{0}(q)$ and we have set $\alpha:=2\cos(\tf{2\pi}{7})$, $\beta:=2\cos(\tf{4\pi}{7})$,$\gamma:=2\cos(\tf{6\pi}{7})$, $\mu:=\tf{1-\sqrt{5}}{2}$, $\lambda:=\tf{1+\sqrt{5}}{2}$. The notation $\langle A_1, A_2, \ldots, A_n \rangle $ is used to denote the group generated by $A_1, \ldots, A_n$ with $|A_1| \geq |A_2| \geq \ldots \geq |A_n|$, where $|A|$ denotes the order of $A$. Here and throughout, we use $C_n$ to denote the cyclic group of order $n$.
	Table \ref{sunwilmc} lists multiplicative combinations that apply directly to the examples appearing in later sections. We remind the reader that the property of being an eigenform for all Hecke operators is a stronger condition than multiplicativity.\\

In \cite{shoe}, Schoeneberg proved the identity
\begin{equation}
   \label{serredif}
   \df{B\left(6,1,\tf{k+1}{24},q\right) - B\left(6,5,\tf{k+25}{24},q\right)}{2} =q^{\tf{k+1}{24}}E(q)E(q^{k}),
   \end{equation}
   for $k \equiv -1\pmod{24}$.\\
	In section \ref{serre} we state and prove an extra parameter generalization of \eqref{serredif}. This generalization is Theorem \ref{6mthm}. Theorem \ref{6mthm} provides a representation of the most general \textit{eta}-products of weight 1, by the difference of two theta series arising from binary quadratic forms. We then state and prove Theorem \ref{psipsithm}, Theorem \ref{4mthm}, and Theorem \ref{9mthm} which provide representations for two families of weight 1 \textit{eta}-quotients as a difference of theta series. We would like to point out that many special cases of Theorem \ref{6mthm} were previously discussed in the literature. (See \cite{berk2}, \cite{blij}, \cite{fricke}, \cite{okamoto}, \cite{shoe}, \cite{serre}, \cite{sun}, \cite{sunwil2}.)\\
	
	The proofs of Theorems \ref{6mthm}, \ref{psipsithm}, \ref{4mthm}, and \ref{9mthm} are all elementary, in the sense that they make essential use of dissections of $q$-series. Lastly, we remark that the \textit{eta}-products arising from Theorem \ref{6mthm} will always yield cusp forms, while the \textit{eta}-quotients arising from Theorem \ref{psipsithm}, Theorem \ref{4mthm}, and \ref{9mthm} are not generally cusp forms.\\
   
   In sections \ref{examples} and \ref{examples71} we employ Theorem \ref{6mthm} and give explicit formulas for the Fourier coefficients of \textit{eta}-products of level 47 and 71.
	In section \ref{-135} we consider \textit{eta}-quotients of level $135, 648, 1024,$ and $1872,$ and employ Theorems \ref{6mthm}, \ref{psipsithm}, \ref{4mthm}, and \ref{9mthm} to deduce formulas for their Fourier coefficients. Also in section \ref{-135}, we contrast our multiplicative completion of $q^7E(q^{12})E(q^{156})$ with that of Gordon and Hughes (see \cite{gord}). In section \ref{outlook} we give concluding remarks.

   \section{elementary proofs of Theorems \ref{6mthm}, \ref{psipsithm}, \ref{4mthm}, and \ref{9mthm}}
   \label{serre}

\begin{theorem}
\label{6mthm}
If $m,s$ are positive integers with $24s-m>0$, then
\[
\df{B(6m,m,s,q) - B(6m,5m,s+m,q)}{2}=q^{s}E(q^m)E(q^{24s -m}).
\]
\end{theorem}

Without loss of generality we assume that $m$ and $s$ are coprime, and hence $(6m,m,s)$ and\\ $(6m,5m,s+m)$ are primitive forms, but not necessarily reduced. The transformations $(x,y) \rightarrow \left(x + \tf{y}{2} ,-y\right)$ and $(x,y) \rightarrow \left(x + \tf{y}{3} , y\right)$, send $(6m,m,s)$ to $(6m,5m,s+m)$. Using these transformations we see $(6m,m,s)$ and $(6m,5m,s+m)$ are equivalent over the $p$-adic integers for all primes $p$. Equivalence over the $p$-adic integers implies $(6m,m,s)$ and $(6m,5m,s+m)$ share the same genus \cite{cox}, \cite{voight}. Hence $q^{s}E(q^m)E(q^{24s -m})$ is a cusp form with simple congruential properties \cite[p.577]{siegel}. \\

We mention that Theorem \ref{6mthm} is explicitly used in section \ref{examples}, see \ref{diff1}; section \ref{examples71}, see \ref{diff171}; section \ref{-135}, see 
\ref{135sub} and \ref{-1872}. We now proceed with the proof of Theorem \ref{6mthm}.

\begin{proof}
For a fixed $j$ we have
\begin{equation}
\label{1}
\sum_{\substack{x\\ y\equiv j \imod{12}}}q^{6mx^{2} +mxy +sy^{2}} = \sum_{\substack{x\\ y\equiv -j \imod{12}}}q^{6mx^{2} +mxy +sy^{2}} = \sum_{x, y}q^{6mx^{2} +mx(12y+j) +s(12y+j)^{2}}.
\end{equation}
Letting $x \rightarrow -x+y$ and $y \rightarrow -y$ on the right-hand side of \eqref{1} yields,
\begin{equation}
\begin{aligned}
\label{12dis}
\sum_{x,y}q^{j^{2}s -jmx +6mx^{2} -j(24s-m)y+6(24s-m)y^{2} }&= q^{sj^{2}}\left(\sum_{x}q^{m(6x^{2}-jx)}\right)\left(\sum_{y}Q^{(6y^{2}-jy)}\right)\\
&=q^{sj^{2}}f\left(q^{m(6-j)},q^{m(6+j)}\right)f\left(Q^{6-j},Q^{6+j}\right),
\end{aligned}
\end{equation}
with $Q:=q^{24s-m}$.\\
Employing \eqref{1} and \eqref{12dis} we obtain
\begin{equation}
\label{7term1}
\begin{aligned}
B(6m,m,s,q) =& \sum_{j=0}^{11}q^{sj^{2}}f\left(q^{m(6-j)},q^{m(6+j)}\right)f\left(Q^{6-j},Q^{6+j}\right)\\
=&f\left(q^{6m},q^{6m}\right)f\left(Q^{6},Q^{6}\right)\\
&+2\sum_{j=1}^{5}q^{sj^{2}}f\left(q^{m(6-j)},q^{m(6+j)}\right)f\left(Q^{6-j},Q^{6+j}\right)\\
&+q^{36s}f\left(1,q^{12m}\right)f\left(1,Q^{12}\right).
\end{aligned}
\end{equation}

We repeat the above process with the form $(6m,5m,s+m)$ to get the analogous seven term expansion for $B(6m,5m,s+m,q)$:
\begin{equation}
\label{7term2}
\begin{aligned}
B(6m,5m,s+m,q) =& \sum_{j=0}^{11}q^{(s+m)j^{2}}f\left(q^{m(6-5j)},q^{m(6+5j)}\right)f\left(Q^{6-j},Q^{6+j}\right)\\
=&f\left(q^{6m},q^{6m}\right)f\left(Q^{6},Q^{6}\right)\\
&+2\sum_{j=1}^{5}q^{(s+m)j^{2}}f\left(q^{m(6-5j)},q^{m(6+5j)}\right)f\left(Q^{6-j},Q^{6+j}\right)\\
&+q^{36(s+m)}f\left(q^{-24m},q^{36m}\right)f\left(1,Q^{12}\right).
\end{aligned}
\end{equation}
Appropriately applying \eqref{fshift} to the right-hand side of \eqref{7term2}, we write \eqref{7term2} as
\begin{equation}
\label{7term3}
\begin{aligned}
B(6m,5m,s+m,q) =& f\left(q^{6m},q^{6m}\right)f\left(Q^{6},Q^{6}\right)\\
 &+ 2q^{s+m}f\left(q^{m},q^{11m}\right)f\left(Q^{5},Q^{7}\right)\\
&+2\sum_{j=2}^{4}q^{sj^{2}}f\left(q^{m(6-j)},q^{m(6+j)}\right)f\left(Q^{6-j},Q^{6+j}\right)\\
&+2q^{25s-m}f\left(q^{5m},q^{7m}\right)f\left(Q,Q^{11}\right)\\
&+q^{36s}f\left(1,q^{12m}\right)f\left(1,Q^{12}\right).
\end{aligned}
\end{equation}
It is clear that many of the terms in \eqref{7term3} also appear in \eqref{7term1}. Subtracting \eqref{7term3} from \eqref{7term1} we perform numerous cancellations to obtain
\begin{equation}\label{ident1}
\begin{split}
\df{B(6m,m,s,q) - B(6m,5m,s+m,q)}{2}=&~q^{s}f\left(q^{5m},q^{7m}\right)f\left(Q^{5},Q^{7}\right)+q^{25s}f\left(q^{m},q^{11m}\right)f\left(Q,Q^{11}\right)\\
-&q^{s+m}f\left(q^{m},q^{11m}\right)f\left(Q^{5},Q^{7}\right)-q^{s}Qf\left(q^{5m},q^{7m}\right)f\left(Q,Q^{11}\right).
\end{split}
\end{equation}
The four terms on the right-hand side of \eqref{ident1} can be written as the product
\begin{equation}
\label{factor2}
q^{s}\cdot\left(f\left(q^{5m},q^{7m}\right)-q^{m}f\left(q^{m},q^{11m}\right)\right)\cdot\left(f\left(Q^{5},Q^{7}\right)-Qf\left(Q,Q^{11}\right)\right).
\end{equation}
Employing \eqref{pentcor}, we conclude that
\begin{equation}
\df{B(6m,m,s,q) - B(6m,5m,s+m,q)}{2}=q^{s}E(q^m)E(q^{24s -m}).
\end{equation}
\end{proof}

We remark that a result of Sun given in \cite{sun} is closely related to our Theorem \ref{6mthm}. Sun's result makes essential use of a lemma, found in \cite[p.~356]{sunwil2}, which employs representations by four binary quadratic forms. It can be shown that the result of \cite[(p.~16, Thm.~3.1)]{sun} follows as a corollary to our Theorem \ref{6mthm}.

\begin{theorem}
\label{psipsithm}
If $m,s$ are positive integers with $8s-m>0$, then
\begin{align*}
\label{psipsi}
\df{B(8m,2m,s,q) - B(8m,6m,s+m,q)}{2}&=q^s\df{E(q^m)E(q^{4m})E(q^{8s-m})E(q^{4(8s-m)})}{E(q^{2m})E(q^{2(8s-m)})}\\
&=q^{s}\psi(-q^{8s-m})\psi(-q^{m}).
\end{align*}
\end{theorem}

We give an application of Theorem \ref{psipsithm} in section \ref{-135}. Specifically, the example involving level $1024$ (see \ref{1024}) makes explicit use of Theorem \ref{psipsithm}.	The proof of Theorem \ref{psipsithm} is similar to the proof of Theorem \ref{6mthm}, and so we sketch the proof.

\begin{proof}

One can show that we have the 5 term expansions
\begin{equation}
\label{7term18}
\begin{aligned}
B(8m,2m,s,q) =& \sum_{j=0}^{7}q^{sj^{2}}f\left(q^{m(8-2j)},q^{m(8+2j)}\right)f\left(Q^{8-2j},Q^{8+2j}\right)\\
=&f\left(q^{8m},q^{8m}\right)f\left(Q^{8},Q^{8}\right)\\
&+2q^{s}f\left(q^{6m},q^{10m}\right)f\left(Q^{6},Q^{10}\right)\\
&+2q^{4s}f\left(q^{4m},q^{12m}\right)f\left(Q^4,Q^{12}\right)\\
&+2q^{9s}f\left(q^{2m},q^{14m}\right)f\left(Q^{2},Q^{14}\right)\\
&+q^{16s}f\left(1,q^{16m}\right)f\left(1,Q^{16}\right),
\end{aligned}
\end{equation}

and
\begin{equation}
\label{7term38}
\begin{aligned}
B(8m,6m,s+m,q) =& \sum_{j=0}^{7}q^{(s+m)j^{2}}f\left(q^{m(8-6j)},q^{m(8+6j)}\right)f\left(Q^{8-2j},Q^{8+2j}\right)\\
=&f\left(q^{8m},q^{8m}\right)f\left(Q^{8},Q^{8}\right)\\
&+2q^{(s+m)}f\left(q^{2m},q^{14m}\right)f\left(Q^{6},Q^{10}\right)\\
&+2q^{4s}f\left(q^{4m},q^{12m}\right)f\left(Q^{4},Q^{12}\right)\\
&+2q^{9s-m}f\left(q^{6m},q^{10m}\right)f\left(Q^{2},Q^{14}\right)\\
&+q^{16s}f\left(1,q^{16m}\right)f\left(1,Q^{16}\right),
\end{aligned}
\end{equation}
where $Q:=q^{8s-m}$. Subtracting \eqref{7term38} from \eqref{7term18} we obtain
\begin{equation}\label{ident81}
\setlength{\jot}{10pt}
\begin{split}
&B(8m,2m,s,q) - B(8m,6m,s+m,q)\\
&=2q^{s}\cdot\left(f\left(q^{6m},q^{10m}\right)-q^{m}f\left(q^{2m},q^{14m}\right)\right)\cdot\left(f\left(Q^{6},Q^{10}\right)-Qf\left(Q^2,Q^{14}\right)\right).
\end{split}
\end{equation}

Employing \eqref{psipsiaux}, we conclude that
\begin{equation}
\df{B(8m,2m,s,q) - B(8m,6m,s+m,q)}{2}=q^{s}\psi(-q^{8s-m})\psi(-q^{m}).
\end{equation}
\end{proof}

\begin{theorem}
\label{4mthm}
If $m,k$ are positive integers, then
\begin{align*}
\df{B\left(m,0,4k,q\right) - B\left(4m,4m,k+m,q\right)}{2}&=q^m\df{E^2(q^s)E^2(q^{16m})}{E(q^{2s})E(q^{8m})}\\
&=q^m\psi(q^{8m})\phi\left(-q^{k}\right).
\end{align*}
\end{theorem}

We give an application of Theorem \ref{4mthm} in section \ref{-135}. The example involving level $1024$ (see \ref{1024}) makes explicit use of both Theorem \ref{psipsithm} and Theorem \ref{4mthm}.\\

   \begin{proof}
   We have
   \begin{align*}
   B(m,0,4k,q) - B(4m,4m,k+m,q)   &=   \sum_{\substack{x\\ y \equiv 0\imod{2} }}q^{mx^2 +ky^2} - \sum_{x,y}q^{m(2x+y)^2 +ky^2}\\
   &=   \sum_{\substack{x\\ y \equiv 0\imod{2} }}q^{mx^2 +ky^2} - \sum_{x \equiv y\imod{2}}q^{mx^2 +ky^2}\\
   &=   \sum_{\substack{x \equiv 1 \imod{2}\\ y \equiv 0 \imod{2} }}q^{mx^2 +ky^2} - \sum_{x \equiv y \equiv 1 \imod{2}}q^{mx^2 +ky^2}\\
   &=   \sum_{x \equiv 1 \imod{2}}q^{mx^2} \left( \sum_{ y \equiv 0 \imod{2}}q^{ky^2}-\sum_{ y \equiv 1 \imod{2}}q^{ky^2}\right)\\
   &=   \left(\sum_{x}q^{m(2x+1)^2}\right) \left( \sum_{ y}(-1)^y q^{ky^2}\right)\\
   &= 2q^m\psi(q^{8m})\phi(-q^k).
   \end{align*}

   \end{proof}

	\begin{theorem}
	\label{9mthm}
If $m,s$ are positive integers, then
\begin{align*}
\df{B(m,0,9s,q) - B(9m,6m,s+m,q)}{2}&=q^m\df{E(q^s)E(q^{4s})E^2(q^{6s})}{E(q^{2s})E(q^{3s})E(q^{12s})}\cdot \df{E^2(q^{6m})E(q^{9m})E(q^{36m})}{E(q^{3m})E(q^{12m})E(q^{18m})}.
\end{align*}
\end{theorem}

\begin{proof}
We employ \eqref{mod31} to find
\begin{equation}
\label{9m1}
B(m,0,9s,q)=\phi(q^{9m})\phi(q^{9s})+ 2q^m f(q^{3m},q^{15m})\phi(q^{9s}).
\end{equation}
We also have
\begin{align*}
\label{9m2}
B(9m,6m,s+m,q)&= \sum_{x, 3\mid y}q^{m(3x+y)^2+sy^2}+\sum_{x, 3\nmid y}q^{m(3x+y)^2+sy^2}\\
&=\sum_{x,y}q^{9m(x+y)^2+9sy^2}+2\sum_{x, y}q^{m(3x+3y+1)^2+s(3y+1)^2}\\
&= \phi(q^{9m})\phi(q^{9s}) +2q^{s+m}\sum_{x, y}q^{6mx+9mx^2+6sy+9sy^2}\\
&= \phi(q^{9m})\phi(q^{9s}) +2q^{s+m}f(q^{3m},q^{15m})f(q^{3s},q^{15s}).
\end{align*}
Thus we come to
\begin{equation}
\label{9a}
\df{B(m,0,9s,q) - B(9m,6m,s+m,q)}{2}=q^mf(q^{3m},q^{15m})\left(\phi(q^{9s})-q^sf(q^{3s},q^{15s})\right).
\end{equation}
Letting $q \rightarrow -q^s$ in \eqref{mod32} yields
\begin{equation}
\label{psiauxxx}
\psi(-q^s)= f(-q^{3s},q^{6s}) -q^s\psi(-q^{9s}).
\end{equation}

Multiplying both sides of \eqref{psiauxxx} by $\tf{E^2(q^{6s})}{E(q^{3s})E(q^{12s})}$ gives
\begin{equation}
\label{auxxpsi}
\df{E(q^s)E(q^{4s})E^2(q^{6s})}{E(q^{2s})E(q^{3s})E(q^{12s})}=\phi(q^{9s})-q^sf(q^{3s},q^{15s}).
\end{equation}
Letting $q \rightarrow q^{3m}$ in \eqref{f15} and employing \eqref{auxxpsi}, allows us to write the right-hand side of \eqref{9a} as an \textit{eta}-quotient. Hence the proof is complete.

\end{proof}

   We remark that unlike Theorem \ref{6mthm}, the forms of Theorem \ref{psipsithm}, Theorem \ref{4mthm}, and Theorem \ref{9mthm} are not necessarily in the same genus.

   \section{Weight 1 Eta-Product of Level $47$}
   \label{examples}

        In this section we determine the Fourier coefficients of $q^{2}E(q)E(q^{47})$. 
 We have CL$(-47) \cong C_5$ and the reduced forms are
\begin{center}
\begin{tabular}{ | l | l | l | }
  \hline     
  \multicolumn{2}{|c|}{CL$(-47) \cong C_5 $}& $\left(\tf{-47}{p}\right)$\\
  \hline                   
  Principal Genus & $(1,1,12), (2,\pm1,6), (3,\pm1,4)$ &+1 \\ \hline 
\end{tabular}
\begin{flushright}
			 .
			 \end{flushright}
   \end{center}
   In the above table, $p$ is taken to be coprime to $-47$ and represented by the given genus.
	
   Taking $m=1$ and $s=2$ in Theorem \ref{6mthm} yields
   \begin{equation}
   \label{diff1}
   \sum_{n>0} a(n)q^n:=\df{B(6,1,2,q) - B(6,5,3,q)}{2}=\df{B(2,1,6,q) - B(3,1,4,q)}{2}=q^{2}E(q)E(q^{47}).
   \end{equation}
   One can check $q^{2}E(q)E(q^{47})$ is not an eigenform for all Hecke operators; hence, we must find a completion for $q^{2}E(q)E(q^{47})$. Throughout this examples we will let $\mu:=\tf{1-\sqrt{5}}{2}$ and $\lambda:=\tf{1+\sqrt{5}}{2}$. Using the first and second row of Table \ref{sunwilmc}, with $A=(2,1,6)$, we find
      
   \begin{equation}
	\label{47a1}
   A_1(q) :=\df{B(1,1,12,q) -\mu B(2,1,6,q)-\lambda B(3,1,4,q)}{2},
   \end{equation}
   and
   \begin{equation}
	\label{47a2}
  A_2(q) :=\df{B(1,1,12,q) -\lambda B(2,1,6,q)-\mu B(3,1,4,q)}{2}
   \end{equation}
	are multiplicative. We show \eqref{47a1} and \eqref{47a2} are eigenforms for all Hecke operators, and we find their eigenvalues. We then derive a formula for the coefficients of $A_1(q)$ and $A_2(q)$, and exploit
   \begin{equation}
   \label{47diffrel}
   [q^n]q^{2}E(q)E(q^{47}) = \df{[q^n]A_{1}(q)-[q^n]A_2(q)}{\sqrt{5}},
   \end{equation}
   to find the Fourier coefficients of $q^{2}E(q)E(q^{47})$.\\

   To show \eqref{47a1} and \eqref{47a2} are eigenforms for $T_p$, we consider the action of $T_p$ on the forms of discriminant $-47$. We separate cases according to the value of $\left(\tf{-47}{p}\right)$. 
	\textbf{
\begin{flushleft}
Case 1: $\left(\df{-47}{p}\right)=1$.
\end{flushleft}
}
Table \ref{47tab1} gives the explicit action of $T_p$ on the theta series associated with forms of discriminant $-47$.

   \begin{table}[htb]  
	\caption{} \label{47tab1}
     \begin{center} 
       \begin{tabular}{|l|c|c|c|}              \hline 
         $B(a,b,c,q)$ & $(1,1,12,p)>0$ & $(3,1,4,p)>0$ & $(2,1,6,p)>0$ \\ \hline
         $B(1,1,12,q) \xrightarrow{T_p}$   & $2B(1,1,12,q)$    & $2B(3,1,4,q)$ & $2B(2,1,6,q)$            \\ \hline
         $B(3,1,4,q) \xrightarrow{T_p}$     & $2B(3,1,4,q)$    & $B(1,1,12,q) + B(2,1,6,q)$ & $B(3,1,4,q) + B(2,1,6,q)$            \\ \hline
         $B(2,1,6,q) \xrightarrow{T_p}$      & $2B(2,1,6,q)$   & $B(3,1,4,q) + B(2,1,6,q)$  & $B(1,1,12,q) + B(3,1,4,q)$    \\ \hline
       \end{tabular}
			\begin{flushright}
			 .
			 \end{flushright}
     \end{center}
   \end{table}
\noindent
	We comment that Table \ref{47tab1} is consistent with the formulas of Hecke \cite[p.794]{hecke}.
	
	\break
	
	\noindent
   With the aid of Table \ref{47tab1}, we compute the action of $T_p$ on $A_1(q)$ and $A_2(q)$: 
	
   \begin{table}[htb]  
	\caption{} \label{47tab2}
     \begin{center} 
       \begin{tabular}{|l|c|c|c|}              \hline 
         $ $ & $(1,1,12,p)>0$ & $(3,1,4,p)>0$ & $(2,1,6,p)>0$ \\ \hline
         $A_1(q) \xrightarrow{T_p}$   & $2A_1(q)$    & $-\lambda A_1(q)$ & $-\mu A_1(q)$            \\ \hline
         $A_2(q)\xrightarrow{T_p}$     & $2A_2(q)$    & $-\mu A_2(q)$ & $-\lambda A_2(q)$            \\ \hline
       \end{tabular}
       \label{tptab247}
			\begin{flushright}
			 .
			 \end{flushright}
     \end{center}
   \end{table}
	
	\noindent
   Hence, \eqref{47a1} and \eqref{47a2} are eigenforms for $T_p$, with eigenvalues 2, $-\mu$, $-\lambda$ when $\left(\tf{-47}{p}\right)=1$.\\
      \textbf{
\begin{flushleft}
Case 2: $\left(\df{-47}{p}\right)=-1$.
\end{flushleft}
}

\noindent
A prime $p$ with $\left(\tf{-47}{p}\right)=-1$ implies $(a,b,c,p)= 0$ for any form of discriminant $-47$. Thus, \eqref{47a1} and \eqref{47a2} are eigenforms for such $T_p$ with eigenvalue 0.
 \textbf{
\begin{flushleft}
Case 3: $\left(\df{-47}{p}\right)=0$.
\end{flushleft}
}
    \noindent
		We find that \eqref{47a1} and \eqref{47a2} are eigenforms for $T_{47}$ with eigenvalue 1.\\
   \noindent
   We have shown \eqref{47a1} and \eqref{47a2} are eigenforms for all Hecke operators, and have found their corresponding eigenvalues. \\

    We now state criteria to determine when $(a,b,c,p)>0$ for a form of discriminant $-47$. We are able to distinguish each case by examining the Weber Class polynomial for discriminant $-47$, $W_{-47}(z):= z^5 +2z^4 +2z^3 +z^2-1$. We only consider primes $p$ with $\left(\tf{-47}{p}\right)\neq-1$. As mentioned in section \ref{intro}, we follow \cite{cox} and \cite{voight} to find:

\begin{enumerate}

      \item $(1,1,12,p)>0$ iff $p=47$ or rem$(z^p,W_{-47}(z)) \equiv z \imod{p}$,
     
     \vspace{.3cm}
     
     \item  $(2,1,6,p)>0$ iff rem$(94z^p,W_{-47}(z)) \equiv $\\ $(-47+3r_p)z^4+(-5r_p-47)z^3+(-11r_p-47)z^2+(-5r_p-47)z-47-r_p \imod{p}$,

      \vspace{.3cm}
     
     \item  $(3,1,4,p)$ iff rem$(94z^p,W_{-47}(z)) \equiv$ \\ $(r_p+47)z^4+(47-r_p)z^3+(7r_p+47)z^2+12zr_p-47+5r_p \imod{p}$,
     
     \end{enumerate}
     
      \vspace{.4cm}
			\noindent
     where $r_p$ is defined by $(r_p)^2 \equiv -47 \imod{p}$.\\
     
     Define $S_1,S_2,S_3$ to be the set of primes $p\neq 47$ represented by $(1,1,12), (2,1,6),$ and $(3,1,4)$, respectively. Let $S_4$ be the set of primes $p$ with $\left(\tf{-47}{p}\right)=-1$.\\
    We factor a positive integer $n$ as
   \begin{equation}
   \label{nfactor47}
   n= 47^{\operatorname{ord}_{47}(n)}\prod_{p_1 \in S_1}p_1^{\operatorname{ord}_{p_1}(n)}\prod_{p_2 \in S_2}p_2^{\operatorname{ord}_{p_2}(n)}\prod_{p_3 \in S_3}p_3^{\operatorname{ord}_{p_3}(n)}\prod_{p_4 \in S_4}p_4^{\operatorname{ord}_{p_4}(n)}.
   \end{equation}

   Employing \eqref{heckedef4}, along with our computed eigenvalues, we obtain
   
   \begin{equation}
   \label{a1piece}
[q^{p^{\nu}}]A_1(q) = \left\{ \begin{array}{ll}
       1 &  p=47,\\
       1+\nu &  p \in S_1, \\
       U(\nu) &  p \in S_2, \\
       V(\nu) &  p \in S_3, \\
       \frac{(-1)^{\nu}+1}{2} &  p \in S_4,
     \end{array}
     \right.
\end{equation}
and
      \begin{equation}
   \label{a2piece}
[q^{p^{\nu}}]A_2(q) = \left\{ \begin{array}{ll}
       1 &  p=47,\\
       1+\nu &  p \in S_1, \\
       V(\nu) &  p \in S_2, \\
       U(\nu) &  p \in S_3, \\
       \frac{(-1)^{\nu}+1}{2} &  p \in S_4,
     \end{array}
     \right.
\end{equation}
where

   \begin{equation}
   U(n):=\df{\sin(2\pi(n+1)/5)}{\sin(2\pi/5)},
   \end{equation}
   and
   \begin{equation}
   V(n):=\df{\sin(4\pi(n+1)/5)}{\sin(4\pi/5)},
   \end{equation}
   are arithmetic functions of period 5. Using the multiplicativity of $A_1(q)$ and $A_2(q)$, along with \eqref{a1piece} and \eqref{a2piece}, we have
   \begin{equation}
   \label{a1def}
   [q^n]A_{1}(q)= \Delta(n)\prod_{p_2 \in S_2}U({\operatorname{ord}_{p_2}(n)})\prod_{p_3 \in S_3}V({\operatorname{ord}_{p_3}(n)}),
   \end{equation}
   and
   \begin{equation}
   \label{a2def}
   [q^n]A_{2}(q)= \Delta(n)\prod_{p_2 \in S_2}V({\operatorname{ord}_{p_2}(n)})\prod_{p_3 \in S_3}U({\operatorname{ord}_{p_3}(n)}),
   \end{equation}
   where 
	\[
	\Delta(n):= \prod_{p_1 \in S_1}(1+{\operatorname{ord}_{p_1}(n)})\prod_{p_4 \in S_4}\df{1+(-1)^{\operatorname{ord}_{p_4}(n)}}{2}.
	\]  
	Employing \eqref{47diffrel}, \eqref{a1def}, and \eqref{a2def}, we obtain
  
  \begin{equation}
  \label{aform1}
   a(n) = \Delta(n)\tf{P(n)}{\sqrt{5}},
   \end{equation}
   where
   \begin{equation}
   \label{pin}
   P(n) =\prod_{p_2 \in S_2}U(\operatorname{ord}_{p_2}(n))\prod_{p_3 \in S_3}V(\operatorname{ord}_{p_3}(n)) - \prod_{p_2 \in S_2}V(\operatorname{ord}_{p_2}(n))\prod_{p_3 \in S_3}U(\operatorname{ord}_{p_3}(n)).
   \end{equation}
	We explore the functions $U(n)$ and $V(n)$, and find a different representation for both $P(n)$ and $a(n)$. Since $U(n)$ and $V(n)$ are periodic, we can tabulate their explicit values
   
   \begin{table}[htb]  
	\caption{} \label{47tab3}
     \begin{center} 
       \begin{tabular}{|l|c|c|}              \hline 
         $n$ & $U(n)$ & $V(n)$ \\ \hline
         $n \equiv 0 \pmod{5}$   & 1    & 1            \\ \hline
         $n \equiv 1 \pmod{5}$     & $-\mu$    & $-\lambda$           \\ \hline
         $n \equiv 2 \pmod{5}$      & $\mu$    & $\lambda$      \\ \hline
         $n \equiv 3 \pmod{5}$      & $-1$    & $-1$     \\ \hline
         $n \equiv 4 \pmod{5}$      & 0    & 0     \\ \hline
       \end{tabular}
			\begin{flushright}
			 .
			 \end{flushright}
     \end{center}
   \end{table}
   
   Given a positive integer $n$, we define $r_{i}$ to be the number of primes $p_2 \in S_2$ such that $\operatorname{ord}_{p_2}(n)\equiv i\pmod{5}$. Similarly, $s_{i}$ is the number of primes $p_3 \in S_3$ such that $\operatorname{ord}_{p_3}(n)\equiv i\pmod{5}$.\\
   
   \noindent
   We have
   \begin{equation}
   \prod_{p_2 \in S_2}U(\operatorname{ord}_{p_2}(n)) =\delta_{r_4,0} \cdot (-1)^{r_1+r_3}\cdot\mu^{r_1+r_2},
   \end{equation}
	and
   \begin{equation}
   \prod_{p_3 \in S_3}V(\operatorname{ord}_{p_3}(n)) = \delta_{s_4,0}\cdot(-1)^{s_1+s_3}\cdot\lambda^{s_1+s_2},
   \end{equation}
   where $\delta_{i,j}$ is 1 if $i=j$ and 0 otherwise.
   
   Then
   \begin{equation}
   \label{calc1}
   \prod_{p_2 \in S_2}U(\operatorname{ord}_{p_2}(n))\prod_{p_3 \in S_3}V(\operatorname{ord}_{p_3}(n)) =\delta_{r_4+s_4,0}\cdot (-1)^{r_2+r_3+s_1+s_3}\cdot\lambda^{s_1+s_2-r_1-r_2},
   \end{equation}
   and
   \begin{equation}
   \label{calc2}
   \prod_{p_2 \in S_2}V(\operatorname{ord}_{p_2}(n))\prod_{p_3 \in S_3}U(\operatorname{ord}_{p_3}(n)) = \delta_{r_4+s_4,0}\cdot(-1)^{r_2+r_3+s_1+s_3}\cdot\mu^{s_1+s_2-r_1-r_2}.
   \end{equation}
   
   \noindent
   Subtracting \eqref{calc2} from \eqref{calc1} yields
   \begin{equation}
   \label{calc3}
   P(n)= (-1)^{r_2+r_3+s_1+s_3}\cdot\delta_{r_4+s_4,0}\cdot\left(\lambda^{s_1+s_2-r_1-r_2}-\mu^{s_1+s_2-r_1-r_2}\right),
   \end{equation}
	and plugging \eqref{calc3} into \eqref{aform1} yields an alternate representation for $a(n)$.

	We can simplify further by employing the identity
	\begin{equation}
	\label{sayw}
	\df{\lambda^x - \mu^x}{\sqrt{5}} = b(x-1),
	\end{equation}
	where 
	\[
	b(L):= \sum_{j=-L}^L(-1)^j\binom{L}{\lceil \tf{L+5j}{2} \rceil},
	\]
	$\lceil \cdot \rceil$ is the ceiling function, and $x$ is an integer greater than $0$. For the case $x< 0$ we can use
	\[
	\df{\lambda^{-|x|} - \mu^{-|x|}}{\sqrt{5}} = (-1)^{x+1}\df{\lambda^{|x|} - \mu^{|x|}}{\sqrt{5}},
	\]
	in order to employ \eqref{sayw}. Noting that $x=0$ implies $\tf{\lambda^x - \mu^x}{\sqrt{5}}=0$, we come to
	\begin{equation}
	a(n) = \left\{ \begin{array}{ll}
       \delta_{r_4+s_4,0}\df{}{}\cdot(-1)^{r_2+r_3+s_1+s_3}\cdot \Delta(n)\cdot b(|s_1+s_2-r_1-r_2|-1) & s_1+s_2-r_1-r_2>0 ,\\
       \delta_{r_4+s_4,0}\df{}{}\cdot(-1)^{r_1+r_3+s_2+s_3+1}\cdot \Delta(n)\cdot b(|s_1+s_2-r_1-r_2|-1) & s_1+s_2-r_1-r_2<0 ,\\
       0& s_1+s_2-r_1-r_2=0. \\
     \end{array}
     \right.
	\end{equation}

   \section{Weight 1 Eta-Product of Level $71$}
	
   \label{examples71}

        In this section we determine the Fourier coefficients of $q^{3}E(q)E(q^{71})$. 
 We have CL$(-71) \cong C_7$ and the reduced forms are
\begin{center}
\begin{tabular}{ | l | l | l | }
  \hline     
  \multicolumn{2}{|c|}{CL$(-71) \cong C_7 $}& $\left(\tf{-71}{p}\right)$\\
  \hline                   
  Principal Genus & $(1,1,18),(2,\pm 1,9),(3,\pm 1,6),(4,\pm 3,5)$ &+1 \\ \hline 
\end{tabular}
\begin{flushright}
			 .
			 \end{flushright}
   \end{center}
   In the above table, $p$ is taken to be coprime to $-71$ and represented by the given genus.
	
   Taking $m=1$ and $s=3$ in Theorem \ref{6mthm} yields
   \begin{equation}
   \label{diff171}
   \sum_{n>0} a(n)q^n:=\df{B(6,1,3,q) - B(6,5,4,q)}{2}=\df{B(3,1,6,q) - B(4,3,5,q)}{2}=q^{3}E(q)E(q^{71}).
   \end{equation}
  Note that $q^{3}E(q)E(q^{71})$ is not multiplicative. Throughout this example we will let  $\alpha:=2\cos(\tf{2\pi}{7})$, $\beta:=2\cos(\tf{4\pi}{7})$,$\gamma:=2\cos(\tf{6\pi}{7})$. Using Table \ref{sunwilmc}, we find
      
   \begin{equation}
	\label{71a1}
   A_1(q):=\df{B(1,1,18,q) +\gamma B(2,1,9,q)+\alpha B(4,3,5,q)+\beta B(3,1,6,q)}{2},
   \end{equation}
   \begin{equation}
	\label{71a2}
   A_2(q):=\df{B(1,1,18,q) +\alpha B(2,1,9,q)+\beta B(4,3,5,q)+\gamma B(3,1,6,q)}{2},
   \end{equation}
	and
	\begin{equation}
	\label{71a3}
   A_3(q):=\df{B(1,1,18,q) +\beta B(2,1,9,q)+\gamma B(4,3,5,q)+\alpha B(3,1,6,q)}{2},
   \end{equation}
	are multiplicative.	Note that $\alpha, \beta, \gamma$ are the roots of $F(x):=x^3 +x^2 -2x -1$.
	
	
	
	

	We show $A_1(q), A_2(q), A_3(q)$ are eigenforms for all Hecke operators, and we find their eigenvalues. We then derive a formula for their coeffcients, and exploit
   \begin{align*}
	&(\beta-\alpha)\cdot \tf{A_1(q)}{7}\\
	+&(\gamma-\beta)\cdot \tf{A_2(q)}{7}\\
	+&(\alpha-\gamma)\cdot \tf{A_3(q)}{7}\\
	=&\sum_{n>0} a(n)q^n = q^{3}E(q)E(q^{71}).
	\end{align*}
   to find the Fourier coefficients of $ q^{3}E(q)E(q^{71})$.\\

   We consider the action of $T_p$ on the associated theta series of forms of discriminant $-71$. We separate cases according to the value of $\left(\tf{-71}{p}\right)$. 
	\textbf{
\begin{flushleft}
Case 1: $\left(\df{-71}{p}\right)=1$.
\end{flushleft}
}
Table \ref{71tab1}  gives the explicit action of $T_p$ on the theta series associated with forms of discriminant $-71$.

   \begin{table}[htb]  
	\caption{} \label{71tab1}
     \begin{center} 
		\scalebox{0.81}{
       \begin{tabular}{|l|c|c|c|c|}              \hline 
         $B(a,b,c,q)$ & $(1,1,18,p)>0$ & $(2,1,9,p)>0$ & $(4,3,5,p)>0$&  $(3,1,6,p)>0$\\ \hline
         $B(1,1,18,q) \xrightarrow{T_p}$   & $2B(1,1,18,q)$    & $2B(2,1,9,q)$ &$2B(4,3,5,q)$  & $2B(3,1,6,q)$            \\ \hline
         $B(2,1,9,q) \xrightarrow{T_p}$     & $2B(2,1,9,q)$    & $B(1,1,18,q) + B(4,3,5,q)$ &$B(2,1,9,q) + B(3,1,6,q)$  &  $B(3,1,6,q) + B(4,3,5,q)$          \\ \hline
				 $B(4,3,5,q) \xrightarrow{T_p}$      & $2B(4,3,5,q)$   & $B(2,1,9,q) + B(3,1,6,q)$  & $B(1,1,18,q) + B(3,1,6,q)$ & $B(2,1,9,q) + B(4,3,5,q)$   \\ \hline
         $B(3,1,6,q) \xrightarrow{T_p}$      & $2B(3,1,6,q)$   & $B(3,1,6,q) + B(4,3,5,q)$  &$B(2,1,9,q) + B(4,3,5,q)$  & $B(1,1,18,q) + B(2,1,9,q)$    \\ \hline
       \end{tabular}
			}
			\begin{flushright}
			 .
			 \end{flushright}
     \end{center}
   \end{table}
	\noindent
	We comment that Table \ref{71tab1} is consistent with the formulas of Hecke \cite[p.794]{hecke}.
	
	\noindent
   With the aid Table \ref{71tab1}, we compute the action of $T_p$ on $A_1(q), A_2(q),A_3(q)$: 
	
   \begin{table}[htb]  
	\caption{} \label{71tab2}
     \begin{center} 
       \begin{tabular}{|l|c|c|c|c|}              \hline 
         $ $ & $(1,1,18,p)>0$ & $(2,1,9,p)>0$ & $(4,3,5,p)>0$   & $(3,1,6,p)>0$ \\ \hline
				$A_1(q) \xrightarrow{T_p}$   & $2A_1(q)$    & $\gamma A_1(q)$ & $\alpha A_1(q)$  & $\beta A_1(q)$            \\ \hline
         $A_2(q) \xrightarrow{T_p}$   & $2A_2(q)$    & $\alpha A_2(q)$ & $\beta A_2(q)$  & $\gamma A_2(q)$            \\ \hline
				$A_3(q) \xrightarrow{T_p}$   & $2A_3(q)$    & $\beta A_3(q)$ & $\gamma A_3(q)$  & $\alpha A_3(q)$            \\ \hline
       \end{tabular}
       \label{tptab271}
			\begin{flushright}
			 .
			 \end{flushright}
     \end{center}
   \end{table}
	
	\noindent
   Hence, $A_1(q), A_2(q),A_3(q)$ are eigenforms for $T_p$, with eigenvalues $2,\alpha,\beta,\gamma$ when $\left(\tf{-71}{p}\right)=1$.\\
      \textbf{
\begin{flushleft}
Case 2: $\left(\df{-71}{p}\right)=-1$.
\end{flushleft}
}

\noindent
A prime $p$ with $\left(\tf{-71}{p}\right)=-1$ implies $(a,b,c,p)= 0$ for any form of discriminant $-71$. Thus, $A_1(q), A_2(q),A_3(q)$ are eigenforms for such $T_p$ with eigenvalue 0.
 \textbf{
\begin{flushleft}
Case 3: $\left(\df{-71}{p}\right)=0$.
\end{flushleft}
}
    \noindent
		We find that $A_1(q), A_2(q),A_3(q)$ are eigenforms for $T_{71}$ with eigenvalue 1.\\
   \noindent
   We have shown $A_1(q), A_2(q),A_3(q)$ are eigenforms for all Hecke operators, and have found their corresponding eigenvalues. \\

    We now state criteria to determine when $(a,b,c,p)>0$ for a form of discriminant $-71$. We are able to distinguish each case by examining the Weber Class polynomial for discriminant $-71$, $W_{-71}(z):= z^7 +z^6 -z^5 -z^4 -z^3 +z^2+2z-1$. We only consider primes $p$ with $\left(\tf{-71}{p}\right)\neq-1$. We follow the method of \cite{cox} and \cite{voight} to find:

\begin{enumerate}

      \item $(1,1,18,p)>0$ iff $p=71$ or rem$(z^p,W_{-71}(z)) \equiv z \pmod{p}$,
     
     \vspace{.3cm}
     
     \item  $(2,1,9,p)>0$ iff rem$(142z^p,W_{-71}(z)) \equiv $\\ \scalebox{.75}{$(2r_p -142)z^6+(6r_p-142)z^5+(-5r_p+71)z^4 +(-6r_p+142)z^3 +(-5r_p+213)z^2 +(-15r_p-71)z -213 +11r_p \pmod{p}$},

      \vspace{.3cm}
     
     \item  $(4,3,5,p)$ iff rem$(142z^p,W_{-71}(z)) \equiv$ \\ $20r_pz^6+16r_pz^5-10r_p z^4 -20r_pz^3 +(-27r_p-71)z^2 +(13r_p-71)z +20r_p \pmod{p}$,
		
		\vspace{.3cm}
     
     \item  $(3,1,6,p)>0$ iff rem$(142z^p,W_{-71}(z)) \equiv $\\ \scalebox{.78}{$(10r_p +142)z^6+(10r_p+142)z^5+(r_p-71)z^4 +(-2r_p-142)z^3 +(-4r_p-142)z^2 +(5r_p+71)z +142 +4r_p \pmod{p}$},

     \end{enumerate}
     
      \vspace{.4cm}
			\noindent
     where $r_p$ is defined by $(r_p)^2 \equiv -71 \pmod{p}$.\\
     
     Define $S_1,S_2,S_3,S_4$ to be the set of primes $p\neq 71$ represented by $(1,1,18), (2,1,9),$ $(4,3,5)$, and $(3,1,6)$, respectively. Let $S_5$ be the set of primes $p$ with $\left(\tf{-71}{p}\right)=-1$.\\
    We factor a positive integer $n$ as
   \begin{equation}
   \label{nfactor71}
   n= 71^{\operatorname{ord}_{71}(n)}\prod_{p_1 \in S_1}p_1^{\operatorname{ord}_{p_1}(n)}\prod_{p_2 \in S_2}p_2^{\operatorname{ord}_{p_2}(n)}\prod_{p_3 \in S_3}p_3^{\operatorname{ord}_{p_3}(n)}\prod_{p_4 \in S_4}p_4^{\operatorname{ord}_{p_4}(n)}\prod_{p_5 \in S_5}p_5^{\operatorname{ord}_{p_5}(n)}.
   \end{equation}

   Letting
	\[
	\Delta(n):=\prod_{p_1 \in S_1}(1+{\operatorname{ord}_{p_1}(n)})\prod_{p_5 \in S_5}\df{1+(-1)^{\operatorname{ord}_{p_5}(n)}}{2},
	\]
	we obtain
	\begin{equation}
   \label{a3def71}
   [q^n]A_{1}(q)=\Delta(n)\prod_{p_2 \in S_2}W({\operatorname{ord}_{p_2}(n)})\prod_{p_3 \in S_3}U({\operatorname{ord}_{p_3}(n)})\prod_{p_4 \in S_4}V({\operatorname{ord}_{p_4}(n)}),
   \end{equation}
   \begin{equation}
   \label{a1def71}
   [q^n]A_{2}(q)=\Delta(n)\prod_{p_2 \in S_2}U({\operatorname{ord}_{p_2}(n)})\prod_{p_3 \in S_3}V({\operatorname{ord}_{p_3}(n)})\prod_{p_4 \in S_4}W({\operatorname{ord}_{p_4}(n)}),
   \end{equation}
  \begin{equation}
   \label{a2def71}
   [q^n]A_{3}(q)=\Delta(n)\prod_{p_2 \in S_2}V({\operatorname{ord}_{p_2}(n)})\prod_{p_3 \in S_3}W({\operatorname{ord}_{p_3}(n)})\prod_{p_4 \in S_4}U({\operatorname{ord}_{p_4}(n)}),
   \end{equation}
   where 
	\begin{equation}
   U(n):=\df{\sin(2\pi(n+1)/7)}{\sin(2\pi/7)},
   \end{equation}
   \begin{equation}
   V(n):=\df{\sin(4\pi(n+1)/7)}{\sin(4\pi/7)},
   \end{equation}
	and
	 \begin{equation}
   W(n):=\df{\sin(6\pi(n+1)/7)}{\sin(6\pi/7)}.
   \end{equation}

	Since $U(n), V(n)$ and $W(n)$ are periodic, we can tabulate their explicit values:
   
   \begin{table}[htb] 
	\caption{} \label{71tab3}
     \begin{center} 
       \begin{tabular}{|l|c|c|c|}              \hline 
         $n$ & $U(n)$ & $V(n)$ & $W(n)$\\ \hline
         $n \equiv 0 \pmod{7}$   & 1    & 1   &1         \\ \hline
         $n \equiv 1 \pmod{7}$     & $\alpha$    & $\beta$ & $\gamma$         \\ \hline
         $n \equiv 2 \pmod{7}$      & $-\gamma^{-1}$    & $-\alpha^{-1}$ & $-\beta^{-1}$      \\ \hline
         $n \equiv 3 \pmod{7}$      & $\gamma^{-1}$    & $\alpha^{-1}$ & $\beta^{-1}$     \\ \hline
         $n \equiv 4 \pmod{7}$      & $-\alpha$    & $-\beta$ & $-\gamma$   \\ \hline
				$n \equiv 5 \pmod{7}$      & $-1$    & $-1$  &$-1$   \\ \hline
				$n \equiv 6 \pmod{7}$      & $0$    & 0 &0    \\ \hline
       \end{tabular}
			\begin{flushright}
			 .
			 \end{flushright}
     \end{center}
   \end{table}

	To ease notation we define
		\[
	\Delta_1(n):=\prod_{p_2 \in S_2}W({\operatorname{ord}_{p_2}(n)})\prod_{p_3 \in S_3}U({\operatorname{ord}_{p_3}(n)})\prod_{p_4 \in S_4}V({\operatorname{ord}_{p_4}(n)}).
	\]
\[
\Delta_2(n) := \prod_{p_2 \in S_2}U({\operatorname{ord}_{p_2}(n)})\prod_{p_3 \in S_3}V({\operatorname{ord}_{p_3}(n)})\prod_{p_4 \in S_4}W({\operatorname{ord}_{p_4}(n)}),
\]
	\[
	\Delta_3(n):=\prod_{p_2 \in S_2}V({\operatorname{ord}_{p_2}(n)})\prod_{p_3 \in S_3}W({\operatorname{ord}_{p_3}(n)})\prod_{p_4 \in S_4}U({\operatorname{ord}_{p_4}(n)}),
	\]

	Thus we come to 
	\[
	a(n)=\tf{\Delta(n)}{7}P(n),
	\]
	where

	\begin{align*}
	P(n)=&(\beta-\alpha)\Delta_1(n)\\
	+&(\gamma-\beta)\Delta_2(n)\\
	+&(\alpha-\gamma)\Delta_3(n).\\
	\end{align*}

	Similar to our approach for discriminant $-47$ we can find another representation of $a(n)$ by using the tabulated values of $U,V$ and $W$.\\

   Given a positive integer $n$, we define $r_{i}$ to be the number of primes $p_2 \in S_2$ such that $\operatorname{ord}_{p_2}(n)\equiv i\pmod{7}$. Similarly, $s_{i}$ is the number of primes $p_3 \in S_3$ such that $\operatorname{ord}_{p_3}(n)\equiv i\pmod{7}$, and $t_{i}$ is the number of primes $p_4 \in S_4$ such that $\operatorname{ord}_{p_4}(n)\equiv i\pmod{7}$. We can now write $\Delta_1,\Delta_2,\Delta_3$ in terms of the new notation.\\
   We use the tabulated values of $U(n),V(n),W(n)$ to find
	\begin{align*}
	 \prod_{p_2 \in S_2}U({\operatorname{ord}_{p_2}(n)}) &= \delta_{r_6,0}(-1)^{r_2+r_4+r_5}\alpha^{r_1+r_4}\gamma^{-r_2-r_3},\\
	\prod_{p_3 \in S_3}V({\operatorname{ord}_{p_3}(n)}) &= \delta_{s_6,0}(-1)^{s_2+s_4+s_5}\beta^{s_1+s_4}\alpha^{-s_2-s_3},\\
	\prod_{p_4 \in S_4}W({\operatorname{ord}_{p_4}(n)}) &= \delta_{t_6,0}(-1)^{t_2+t_4+t_5}\gamma^{t_1+t_4}\beta^{-t_2-t_3},\\
	\end{align*}
	and we note that the symmetry between the three products is the same symmetry seen in the table of values for $U(n),V(n),W(n)$.
   Thus we come to
   \begin{equation}
   \Delta_1(n) =\delta_{r_6+s_6+t_6,0} \cdot (-1)^{r_2+r_4+r_5+s_2+s_4+s_5+t_2+t_4+t_5}\cdot \alpha^{s_1+s_4-t_2-t_3}\beta^{t_1+t_4-r_2-r_3}\gamma^{r_1+r_4-s_2-s_3}.
   \end{equation}
	
	We exploit symmetry to find that $\Delta_2(n)$ is equal to $\Delta_1(n)$ under the permutation $\sigma:=(\alpha,\beta,\gamma)$, and $\Delta_3(n)$ is equal to $\Delta_1(n)$ under the permutation $\sigma^2$. Explicitly
	\begin{align*}
	 \Delta_2(n) &= \sigma(\Delta_1(n)),\\
	\Delta_3(n) &= \sigma^2(\Delta_1(n)),\\
	\end{align*}
   and thus we may write
	\begin{align*}
	P(n)=&(\beta-\alpha)\Delta_1(n)\\
	+\sigma(&(\beta-\alpha)\Delta_1(n))\\
	+\sigma^2(&(\beta-\alpha)\Delta_1(n)).
	\end{align*}
	
	Using the relation
	\[
	a(n)=\tf{\Delta(n)}{7}P(n),
	\]
	along with the formula for $P(n)$, yields a general formula for $a(n)=[q^n]q^{3}E(q)E(q^{71})$.

	We now consider some special values of $n$, and the corresponding formula for $a(n)$.\\	
	Given a prime $p$ which is represented by a form of discriminant $-71$, we list the values $a(p^\nu)$:
   \begin{table}[htb]  
	\caption{} \label{71tab4}
     \begin{center} 
       \begin{tabular}{|l|c|c|c|c|}              \hline 
          & $(1,1,18,p)>0$ & $(2,1,9,p)>0$ & $(4,3,5,p)>0$&  $(3,1,6,p)>0$\\ \hline
         $a(p^\nu)$,~ $\nu \equiv 0 \pmod{7}$   & 0 & 0    & 0   &0         \\ \hline
				 $a(p^\nu)$, $\nu \equiv 1 \pmod{7}$   & 0 & 0    & $-1$   &1         \\ \hline
				 $a(p^\nu)$, $\nu \equiv 2 \pmod{7}$   & 0 & $-1$   & 1  &0         \\ \hline
				 $a(p^\nu)$, $\nu \equiv 3 \pmod{7}$   & 0 & $1$    & $-1$   &0         \\ \hline
			 $a(p^\nu)$, 	$\nu \equiv 4 \pmod{7}$   & 0 & 0    & 1   &$-1$         \\ \hline
				 $a(p^\nu)$, $\nu \equiv 5 \pmod{7}$   & 0 & 0    & 0   &0         \\ \hline
				 $a(p^\nu)$, $\nu \equiv 6 \pmod{7}$   & 0 & 0    & 0   &0         \\ \hline

       \end{tabular}
			\begin{flushright}
			 .
			 \end{flushright}
     \end{center}
   \end{table}

  It is interesting to examine the special case $n= p_1^{\nu_1}p_2^{\nu_2} \ldots p_r^{\nu_r}$ with $p_1,\ldots ,p_r \in S_2$. It can be shown by induction that we have
\begin{equation}
\label{71a}
a(n) =\delta_{r_6,0} (-1)^{r_1+r_3+r_5}G(r_3+r_2,r_1+r_4),
\end{equation}
   where $G(L,M)$ is given by
	\begin{equation}
\label{71b}
	G(L,M) =\sum_{j=-L-M}^{L+M}T(L,M,2+7j)-T(L,M,1+7j),
	\end{equation}
	and the trinomial coefficient $T(L,M,a)$ is defined by
	\begin{equation}
\label{71c}
	\sum_{a=-L-\lceil M/2\rceil}^{L+\lceil M/2\rceil}T(L,M,a)x^a=\df{(x^2+x+1)^L(x+1)^M}{x^{L+\lceil M/2\rceil}}.
	\end{equation}
Analogous formulas can be derived when all the prime divisors of $n$ are contained in a single $S_i$ for $1\leq i \leq 5$. The general formula is more involved and will be discussed elsewhere.

	\vspace{.8cm}

   \section{Weight 1 eta-quotients of level 135, 648, 1024, 1872}
   \label{-135}
	
	\subsection{Eta-Products of Level 135}
	\label{135sub}
	
    Applying Theorem \ref{6mthm} with $m=9, s=1$ yields
   \begin{equation}
   \label{diff1135}
   \df{B(54,9,1,q) - B(54,45,10,q)}{2}=\df{B(1,1,34,q) - B(4,3,9,q)}{2}=qE(q^9)E(q^{15}),
   \end{equation}
   and $m=3, s=2$ gives
   \begin{equation}
   \label{diff2135}
   \df{B(18,3,2,q) - B(18,15,5,q)}{2}=\df{B(2,1,17,q) - B(5,5,8,q)}{2}=q^2E(q^3)E(q^{45}).  
	 \end{equation}
	
 We have CL$(-135) \cong C_6 $ and the reduced forms listed according to genus are
\begin{center}
\begin{tabular}{ | l | l | l | l | }
  \hline     
  \multicolumn{2}{|c|}{CL$(-135) \cong C_6 $}& $\left(\tf{p}{5}\right)$ & $\left(\tf{p}{3}\right)$ \\
  \hline                   
  Principal Genus & $(1,1,34)$, $(4,3,9)$, $(4,-3,9)$ &$+1$ &$+1$ \\ \hline
  Second Genus & $(5,5,8)$, $(2,1,17)$, $(2,-1,17)$ & $-1$ &$-1$ \\ \hline  
\end{tabular}
\begin{flushright}
			 .
			 \end{flushright}
   \end{center}
  In the above table, $p$ is taken to be coprime to $-135$ and represented by the given genus.

   With the aid of Table \ref{sunwilmc} with $A=(2,1,17)$ and $\langle A \rangle \cong C_6$, we find 
   \begin{align}
   \label{mult1135}   
   A(q):= \df{B(1,1,34,q)-B(4,3,9,q)+ B(2,1,17,q)-B(5,5,8,q)}{2}
   \end{align}
   is multiplicative. Note that \eqref{mult1135} is the sum of \eqref{diff1135} and \eqref{diff2135}. Similar to the previous example, we show \eqref{mult1135} is an eigenform for all Hecke operators by examining the action of $T_p$ on the forms of discriminant $-135$.\\
	We remark that both $qE(q^9)E(q^{15}) \pm q^2E(q^3)E(q^{45})$ are eigenforms for all $T_p$, and that $qE(q^9)E(q^{15})$, $q^2E(q^3)E(q^{45})$ are related to the quadratic field $\mathbb{Q}(\sqrt{-15})$, as explained by K$\ddot{\mbox{o}}$hler in \cite[p.252]{kohler}.

    \textbf{
\begin{flushleft}
Case 1: $\left(\df{-135}{p}\right)=1$.
\end{flushleft}
}

Table \ref{135tab1} gives the explicit action of $T_p$ on the theta series associated with forms of discriminant $-135$.

   \begin{table}[htb]  
	\caption{} \label{135tab1}
     \begin{center} 
     \scalebox{0.9}{
       \begin{tabular}{|l|c|c|c|c|}              \hline 
         $B(a,b,c,q)$ & $(1,1,34,p)>0$ & $(4,3,9,p)>0$ & $(2,1,17,p)>0$ & $(5,5,8,p)>0$\\ \hline
         $B(1,1,34,q) \xrightarrow{T_p}$   & $2B(1,1,34,q)$    & $2B(4,3,9,q)$ & $2B(2,1,17,q)$    & $2B(5,5,8,q)$         \\ \hline
         $B(4,3,9,q) \xrightarrow{T_p}$     & $2B(4,3,9,q)$    & $B(1,1,34,q) + B(4,3,9,q)$ & $B(2,1,17,q) + B(5,5,8,q)$   & $2B(2,1,17,q)$ \\ \hline
         $B(2,1,17,q) \xrightarrow{T_p}$      & $2B(2,1,17,q)$   & $B(5,5,8,q) + B(2,1,17,q)$  & $B(1,1,34,q) + B(4,3,9,q)$  & $2B(4,3,9,q)$ \\ \hline
         $B(5,5,8,q) \xrightarrow{T_p}$      & $2B(5,5,8,q)$   & $2B(2,1,17,q)$  & $2B(4,3,9,q)$   & $2B(1,1,34,q)$   \\ \hline
       \end{tabular}
       }
			\begin{flushright}
			 .
			 \end{flushright}
     \end{center}
   \end{table}
\noindent
We comment that Table \ref{135tab1} is consistent with the formulas of Hecke \cite[p.794]{hecke}.

   \noindent
   Using Table \ref{135tab1}, we find the action of $T_p$ on $A(q)$:
	
   \begin{table}[htb]  
	\caption{} \label{135tab2}
     \begin{center} 
       \begin{tabular}{|l|c|c|c|c|}              \hline 
                 & $(1,1,34,p)>0$ & $(4,3,9,p)>0$ & $(2,1,17,p)>0$ & $(5,5,8,p)>0$\\ \hline
         $A(q) \xrightarrow{T_p}$   & $2A(q)$    & $-A(q)$ & $A(q)$    & $-2A(q)$     \\ \hline
       \end{tabular}
			\begin{flushright}
			 .
			 \end{flushright}
       \label{ttab2}
     \end{center}
   \end{table}

         \textbf{
\begin{flushleft}
Case 2: $\left(\df{-135}{p}\right)=-1$.
\end{flushleft}
}
\noindent
   A form $(a,b,c)$ of discriminant $-135$ has $(a,b,c,p)=0$ when $\left(\tf{-135}{p}\right)=-1$, and hence \eqref{mult1135} is an eigenform for such $T_p$ with eigenvalue 0.

            \textbf{
\begin{flushleft}
Case 3: $\left(\df{-135}{p}\right)=0$.
\end{flushleft}
}

The below tables give the explicit action of $T_3$ and $T_5$ on the theta series associated with forms of discriminant $-135$.

\vspace{.4cm}

         \begin{table}[htb] 
				\caption{} \label{135tab3}
				\begin{center}
			 \begin{tabular}{|l|c|} \hline  
									$B(1,1,34,q) \xrightarrow{T_3}$   & $B(1,1,4,q^3)$         \\ \hline
							 $B(4,3,9,q) \xrightarrow{T_3}$     & $B(1,1,4,q^3)$    \\ \hline
         $B(2,1,17,q) \xrightarrow{T_3}$    & $B(2,1,2,q^3)$   \\ \hline
         $B(5,5,8,q) \xrightarrow{T_3}$      & $B(2,1,2,q^3)$     \\ \hline
				\end{tabular}\hspace{1cm}
			 \begin{tabular}{|l|c|} \hline  
								$B(1,1,34,q) \xrightarrow{T_5}$   & $B(5,5,8,q)$         \\ \hline
								$B(4,3,9,q) \xrightarrow{T_5}$     & $B(2,1,17,q)$    \\ \hline
         $B(2,1,17,q) \xrightarrow{T_5}$    & $B(4,3,9,q)$   \\ \hline
         $B(5,5,8,q) \xrightarrow{T_5}$      & $B(1,1,34,q)$    \\ \hline
       \end{tabular}
\end{center}
\end{table}

   Using the above tables, we see that $A(q)$ is an eigenform for $T_5$ with eigenvalue $-1$, and an eigenform for $T_3$ with eigenvalue 0.\\
   We have shown \eqref{mult1135} is an eigenform for all Hecke operators, and have found the corresponding eigenvalues.\\

   We now state criteria to determine when $(a,b,c,p)>0$ for a form of discriminant $-135$. We examine the factorization of the Weber class polynomial $W_{-135}(x)=x^6 - x^3 - 1$ modulo $p$. Following the method of \cite{cox} and \cite{voight}, we find that for a prime $p$ with $\left(\tf{-135}{p}\right)=1$, we have
      \begin{enumerate}
   \item{ $p$ is represented by the form $(1,1,34)$ if and only if $W_{-135}(x)$ splits completely modulo $p$,}\\
   
   \item{ $p$ is represented by the form $(4,3,9)$  if and only if $W_{-135}(x)$ factors into two irreducible cubic polynomials modulo $p$,}\\   
   
   \item{$p$ is represented by the form $(5,5,8)$  if and only if $W_{-135}(x)$ factors into three irreducible quadratic polynomials modulo $p$,}\\   
   
   \item{ $p$ is represented by the form $(2,1,17)$  if and only if $W_{-135}(x)$ remains irreducible modulo $p$.}\\   
   \end{enumerate}
   
   We define $S_1, S_2, S_3, S_4$ to be the set of primes $p\neq 5$ represented by $(1,1,34)$, $(5,5,8)$, $(4,3,9)$, $(2,1,17)$, respectively. We also take $S_5$ to be the set of primes $p$ with $\left(\tf{-135}{p}\right)=-1 $.    
      Employing \eqref{heckedef4} along with the previously computed eigenvalues, we obtain
      
         \begin{equation}
         \label{135alpha}
      [q^{p^\nu}]A(q) = \left\{ \begin{array}{ll}
       0 &  p=3, ~~\nu>0, \\
       (-1)^\nu &  p=5, \\
       1+\nu &  p \in S_1, \\
       (-1)^\nu(1+\nu) &  p \in S_2, \\
       U(\nu) &  p \in S_3,\\
       V(\nu) &  p \in S_4,\\
       \tf{1+(-1)^\nu}{2} &   p\in S_5,
     \end{array}
     \right.
\end{equation}
where
\begin{equation}
   U(n):=\df{\sin(2\pi(n+1)/3)}{\sin(2\pi/3)},
   \end{equation}
   and
   \begin{equation}
   V(n):=\df{\sin(\pi(n+1)/3)}{\sin(\pi/3)}.
   \end{equation}
The functions $U(n)$ and $V(n)$ are periodic and we tabulate their values in Table \ref{si}.

   \begin{table}[htb]  
	\caption{} \label{si}
     \begin{center} 
       \begin{tabular}{|l|c|c|}              \hline 
         $n$ & $U(n)$ & $V(n)$ \\ \hline
         $n \equiv 0 \pmod{6}$   & $1$    & $1$            \\ \hline
         $n \equiv 1 \pmod{6}$     & $-1$    & $1$           \\ \hline
         $n \equiv 2 \pmod{6}$      & $0$    & $0$     \\ \hline
         $n \equiv 3 \pmod{6}$      & $1$    & $-1$     \\ \hline
         $n \equiv 4 \pmod{6}$      & $-1$   & $-1$     \\ \hline
         $n \equiv 5 \pmod{6}$      & $0$    & $0$     \\ \hline
       \end{tabular}
			\begin{flushright}
			 .
			 \end{flushright}
     \end{center}
   \end{table}

   Given a positive integer $n$ we write
     \begin{equation}
  \label{135factorization}
  n=3^a5^b\prod_{p_1 \in S_1}p_1^{\operatorname{ord}_{p_1}(n)}\prod_{p_2 \in S_2}p_2^{\operatorname{ord}_{p_2}(n)}
  \prod_{p_3 \in S_3}p_3^{\operatorname{ord}_{p_3}(n)}\prod_{p_4 \in S_4}p_4^{\operatorname{ord}_{p_4}(n)}\prod_{p_5 \in S_5}p_5^{\operatorname{ord}_{p_5}(n)},
  \end{equation}
  where $a=\operatorname{ord}_{3}(n)$ and $b=\operatorname{ord}_{5}(n)$.
  Given the factorization of $n$ in \eqref{135factorization}, we employ \eqref{135alpha} to find the coefficient $[q^n]A(q)$ to be
  \begin{equation}
  \label{aform135}
   (-1)^{b+t}\cdot\delta_{a,0}\cdot\prod_{p \in S_1 \cup S_2}(1+{\operatorname{ord}_{p}(n)})\prod_{p_3 \in S_3}U(\operatorname{ord}_{p_3}(n))\prod_{p_4 \in S_4}V(\operatorname{ord}_{p_4}(n))\prod_{p_5 \in S_5}\df{1+(-1)^{\operatorname{ord}_{p_5}(n)}}{2},
   \end{equation}
  where $t$ is the number of primes factors of $n$, counting multiplicity, that are contained in $S_2$.\\
  
  Similar to the previous example, we are able to rewrite \eqref{aform135} by exploiting the periodicity of $U(n), V(n)$. Define $r_{i}$ to be the number of primes $p_3 \in S_3$ such that $\operatorname{ord}_{p_3}(n)\equiv i\pmod{3}$, and $s_{i}$ to be the number of primes $p_4 \in S_4$ such that $\operatorname{ord}_{p_4}(n)\equiv i\pmod{6}$.\\
   
   \noindent
   Given the factorization of $n$ in \eqref{135factorization} we have,
   \begin{equation}
   \prod_{p_3 \in S_3}U(\operatorname{ord}_{p_3}(n)) =\delta_{r_2,0}\cdot (-1)^{r_1},
   \end{equation}
   and
   \begin{equation}
   \prod_{p_4 \in S_4}V(\operatorname{ord}_{p_4}(n)) =\delta_{s_2+s_5,0} \cdot(-1)^{s_3+s_4}.
   \end{equation}
   We rewrite \eqref{aform135} as
   \begin{equation}
  \label{aform1352}
   [q^n]A(q) = (-1)^{b+t+r_1+s_3+s_4}\cdot\delta_{a+r_2+s_2+s_5,0}\cdot \prod_{p \in S_1 \cup S_2}(1+{\operatorname{ord}_{p}(n)})\prod_{p_5 \in S_5}\df{1+(-1)^{\operatorname{ord}_{p_5}(n)}}{2}.
   \end{equation}
   \noindent
 \eqref{aform1352} gives the Fourier coefficients of $qE(q^9)E(q^{15}) + q^2E(q^3)E(q^{45})$. We note that $[q^n]qE(q^9)E(q^{15}) \neq 0$ implies $n \equiv 1 \pmod{3}$, and $[q^n]q^2E(q^3)E(q^{45})\neq 0$ implies $n \equiv 2 \pmod{3}$. Thus we can employ congruences to extract the Fourier coefficients of each product from $qE(q^9)E(q^{15}) + q^2E(q^3)E(q^{45})$. We have
     
   \[
[q^n]qE(q^9)E(q^{15})= \left\{ \begin{array}{ll}
       [q^n]A(q) &  n\equiv 1 \imod{3}, \\
       0 &  n\equiv 0,2 \imod{3}, \\
     \end{array}
     \right.
\]
and
   \[
[q^n]q^2E(q^3)E(q^{45})= \left\{ \begin{array}{ll}
       [q^n]A(q) &  n\equiv 2 \imod{3}, \\
       0 &  n\equiv 0,1 \imod{3}. \\
     \end{array}
     \right.
\]

   We can also write the Fourier coefficients of the products in the following form
     \begin{equation}
  \label{135coef1}
   [q^n]2qE(q^9)E(q^{15}) = (1+(-1)^{b+t+s_1+s_3})[q^n]A(q),
   \end{equation}
   and
    \begin{equation}
  \label{135coef2}
   [q^n]2q^2E(q^3)E(q^{45}) = (1-(-1)^{b+t+s_1+s_3})[q^n]A(q).
   \end{equation}

	\subsection{Eta-Products of Level 648}
	\label{648sub}
	
    Applying Theorem \ref{9mthm} with $m=1, s=18$ yields
   \begin{equation}
   \label{diff1648}
   g(q):=\df{B(1,0,162,q) - B(9,6,19,q)}{2}=q\df{E^2(q^6)E(q^{9})E(q^{72})E^2(q^{108})}{E(q^{3})E(q^{12})E(q^{54})E(q^{216})},
   \end{equation}
   and $m=2, s=9$ gives
   \begin{equation}
   \label{diff2648}
   h(q):=\df{B(2,0,81,q) - B(11,10,17,q)}{2}=q^2\df{E(q^9)E^2(q^{12})E^2(q^{54})E(q^{72})}{E(q^{6})E(q^{24})E(q^{27})E(q^{108})}.  
	 \end{equation}
	We remark that both $g(q)$ and $h(q)$ are cusp forms.
 We have CL$(-648) \cong C_6 $ and the reduced forms listed according to genus are
\begin{center}
\begin{tabular}{ | l | l | l | l | }
  \hline     
  \multicolumn{2}{|c|}{CL$(-648) \cong C_6 $}& $\left(\tf{p}{3}\right)$ & $\left(\tf{-2}{p}\right)$ \\
  \hline                   
  Principal Genus & $(1,0,162)$, $(9,6,19)$, $(9,-6,19)$ &$+1$ &$+1$ \\ \hline
  Second Genus & $(2,0,81)$, $(11,10,17)$, $(11,-10,17)$ & $-1$ &$+1$ \\ \hline  
\end{tabular}
\begin{flushright}
			 .
			 \end{flushright}
   \end{center}
  In the above table, $p$ is taken to be coprime to $-648$ and represented by the given genus.

  Let
   \begin{align}
   \label{mult1648}   
   A(q):=g(q)+h(q)= \df{B(1,0,162,q)-B(9,6,19,q)+ B(2,0,81,q)-B(11,10,17,q)}{2}
   \end{align}
  Similar to the previous examples, we show \eqref{mult1648} is an eigenform for all Hecke operators by examining the action of $T_p$ on the forms of discriminant $-648$.\\

    \textbf{
\begin{flushleft}
Case 1: $\left(\df{-648}{p}\right)=1$.
\end{flushleft}
}

Table \ref{648tab1} gives the explicit action of $T_p$ on the theta series associated with forms of discriminant $-648$.

   \begin{table}[htb]  
	\caption{} \label{648tab1}
     \begin{center} 
     \scalebox{0.8}{
       \begin{tabular}{|l|c|c|c|c|}              \hline 
         $B(a,b,c,q)$ & $(1,0,162,p)>0$ & $(9,6,19,p)>0$ & $(2,0,81,p)>0$ & $(11,10,17,p)>0$\\ \hline
         $B(1,0,162,q) \xrightarrow{T_p}$   & $2B(1,0,162,q)$    & $2B(9,6,19,q)$               & $2B(2,0,81,q)$               & $2B(11,10,17,q)$         \\ \hline
         $B(9,6,19,q) \xrightarrow{T_p}$     & $2B(9,6,19,q)$    & $B(1,0,162,q) + B(9,6,19,q)$ & $2B(11,10,17,q)$           & $B(2,0,81,q)+B(11,10,17,q)$ \\ \hline
         $B(2,0,81,q) \xrightarrow{T_p}$      & $2B(2,0,81,q)$   & $2B(11,10,17,q)$             & $2B(1,0,162,q)$                & $2B(9,6,19,q)$ \\ \hline
         $B(11,10,17,q) \xrightarrow{T_p}$    & $2B(11,10,17,q)$ & $B(11,10,17,q)+B(2,0,81,q)$  & $2B(9,6,19,q)$              	 & $B(1,0,162,q)+B(9,6,19,q)$   \\ \hline
       \end{tabular}
       }
			\begin{flushright}
			 .
			 \end{flushright}
     \end{center}
   \end{table}
\noindent
We comment that Table \ref{648tab1} is consistent with the formulas of Hecke.

   \noindent
   Using Table \ref{648tab1}, we find the action of $T_p$ on $A(q)$:
	
   \begin{table}[htb]  
	\caption{} \label{648tab2}
     \begin{center} 
       \begin{tabular}{|l|c|c|c|c|}              \hline 
                 & $(1,0,162,p)>0$ & $(9,6,19,p)>0$ & $(2,0,81,p)>0$ & $(11,10,17,p)>0$\\ \hline
         $A(q) \xrightarrow{T_p}$   & $2A(q)$    & $-A(q)$ & $2A(q)$    & $-A(q)$     \\ \hline
       \end{tabular}
			\begin{flushright}
			 .
			 \end{flushright}
       \label{ttab2}
     \end{center}
   \end{table}

         \textbf{
\begin{flushleft}
Case 2: $\left(\df{-648}{p}\right)=-1$.
\end{flushleft}
}
\noindent
   A form $(a,b,c)$ of discriminant $-648$ has $(a,b,c,p)=0$ when $\left(\tf{-648}{p}\right)=-1$, and hence \eqref{mult1648} is an eigenform for such $T_p$ with eigenvalue 0.

            \textbf{
\begin{flushleft}
Case 3: $\left(\df{-648}{p}\right)=0$.
\end{flushleft}
}

The below tables give the explicit action of $T_2$ and $T_3$ on the theta series associated with forms of discriminant $-648$.

\vspace{.4cm}

         \begin{table}[htb] 
				\caption{} \label{648tab3}
				\begin{center}
			 \begin{tabular}{|l|c|} \hline  
									$B(1,0,162,q) \xrightarrow{T_2}$   & $B(2,0,81,q)$         \\ \hline
							 $B(9,6,19,q) \xrightarrow{T_2}$     & $B(11,10,17,q)$    \\ \hline
         $B(2,0,81,q) \xrightarrow{T_2}$    & $B(1,0,162,q)$   \\ \hline
         $B(11,10,17,q) \xrightarrow{T_2}$      & $B(9,6,19,q)$     \\ \hline
				\end{tabular}\hspace{1cm}
			 \begin{tabular}{|l|c|} \hline  
								$B(1,0,162,q) \xrightarrow{T_3}$   & $B(1,0,18,q^3)$         \\ \hline
								$B(9,6,19,q) \xrightarrow{T_3}$     & $B(1,0,18,q^3)$    \\ \hline
         $B(2,0,81,q) \xrightarrow{T_3}$    & $B(2,0,9,q^3)$   \\ \hline
         $B(11,10,17,q) \xrightarrow{T_3}$      & $B(2,0,9,q^3)$    \\ \hline
       \end{tabular}
\end{center}
\end{table}

   Using the above tables, we see that $A(q)$ is an eigenform for $T_2$ with eigenvalue $1$, and an eigenform for $T_3$ with eigenvalue 0.\\
   We have shown $A(q)$ is an eigenform for all Hecke operators, and have found the corresponding eigenvalues. Lastly we note that the above tables also show $g(q)-h(q)$ is an eigenform for all $T_p$.\\

   We now state criteria to determine when $(a,b,c,p)>0$ for a form of discriminant $-648$. We examine the factorization of the Weber class polynomial $W_{-648}(x)=x^6 - 7758x^5 - 17217x^4 - 25316x^3 - 17217x^2 - 7758x + 1$ modulo $p$. Following the method of \cite{cox} and \cite{voight}, we find that for a prime $p$ with $\left(\tf{-648}{p}\right)=1$, we have
      \begin{enumerate}
   \item{ $p$ is represented by the form $(1,0,162)$ if and only if $W_{-648}(x)$ splits completely modulo $p$,}\\
   
   \item{ $p$ is represented by the form $(9,6,19)$  if and only if $W_{-648}(x)$ factors into two irreducible cubic polynomials modulo $p$,}\\   
   
   \item{$p$ is represented by the form $(2,0,81)$  if and only if $W_{-648}(x)$ factors into three irreducible quadratic polynomials modulo $p$,}\\   
   
   \item{ $p$ is represented by the form $(11,10,17)$  if and only if $W_{-648}(x)$ remains irreducible modulo $p$.}\\   
   \end{enumerate}

   We define $S_1$ to be the set of primes $p\neq 2$ with $(1,0,162,p)+(2,0,81,p)>0$. Similarly, we let $S_2$ be the set of primes $p$ with $(9,6,19,p)+(11,10,17,p)>0$. We also take $S_3$ to be the set of primes $p$ with $\left(\tf{-648}{p}\right)=-1 $.    
      Employing \eqref{heckedef4} along with the previously computed eigenvalues, we obtain
      
         \begin{equation}
         \label{648alpha}
      [q^{p^\nu}]A(q) = \left\{ \begin{array}{ll}
       0 &  p=3, ~~\nu>0, \\
       1 &  p=2, \\
       1+\nu &  p \in S_1, \\
       U(\nu) &  p \in S_2,\\
       \tf{1+(-1)^\nu}{2} &   p\in S_3,
     \end{array}
     \right.
\end{equation}
where
\begin{equation}
   U(n):=\df{\sin(2\pi(n+1)/3)}{\sin(2\pi/3)}.
   \end{equation}
The function $U(n)$ is discussed in \ref{135sub}, and the explicit values of $U(n)$ can be found in Table \ref{si}.

   Given a positive integer $n$ we write
     \begin{equation}
  \label{648factorization}
  n=2^a3^b\prod_{p_1 \in S_1}p_1^{\operatorname{ord}_{p_1}(n)}\prod_{p_2 \in S_2}p_2^{\operatorname{ord}_{p_2}(n)}
  \prod_{p_3 \in S_3}p_3^{\operatorname{ord}_{p_3}(n)},
  \end{equation}
  where $a=\operatorname{ord}_{2}(n)$ and $b=\operatorname{ord}_{3}(n)$.
  Given the factorization of $n$ in \eqref{648factorization}, we employ \eqref{648alpha} to find the coefficient $[q^n]A(q)$ to be
  \begin{equation}
  \label{aform648}
  \delta_{b,0}\cdot\prod_{p_1 \in S_1}(1+{\operatorname{ord}_{p_1}(n)})\prod_{p_2 \in S_2}U(\operatorname{ord}_{p_2}(n))\prod_{p_3 \in S_3}\df{1+(-1)^{\operatorname{ord}_{p_3}(n)}}{2}.
   \end{equation}
  Define $r_{i}$ to be the number of primes $p_2 \in S_2$ such that $\operatorname{ord}_{p_2}(n)\equiv i\pmod{3}$.\\
   
   \noindent
   Given the factorization of $n$ in \eqref{648factorization} we have,
   \begin{equation}
   \prod_{p_2 \in S_2}U(\operatorname{ord}_{p_2}(n)) =\delta_{r_2,0}\cdot (-1)^{r_1},
   \end{equation}
   and we rewrite \eqref{aform648} as
   \begin{equation}
  \label{aform6482}
   [q^n]A(q) = (-1)^{r_1}\cdot\delta_{b+r_2,0}\cdot \prod_{p_1 \in S_1}(1+{\operatorname{ord}_{p_1}(n)})\prod_{p_3 \in S_3}\df{1+(-1)^{\operatorname{ord}_{p_3}(n)}}{2}.
   \end{equation}
   \noindent
 Both \eqref{aform648} and\eqref{aform6482} give a formula for the Fourier coefficients of $A(q)$, and to extract the Fourier coefficients of the cusp forms $g(q)$ and $h(q)$ we employ congruences. Note that $[q^n]g(q)\neq 0$ implies $n \equiv 1 \pmod{3}$, and $[q^n]h(q)\neq 0$ implies $n \equiv 2 \pmod{3}$. We have
     
   \[
[q^n]g(q)= \left\{ \begin{array}{ll}
       [q^n]A(q) &  n\equiv 1 \imod{3}, \\
       0 &  n\equiv 0,2 \imod{3}, \\
     \end{array}
     \right.
\]
and
   \[
[q^n]h(q)= \left\{ \begin{array}{ll}
       [q^n]A(q) &  n\equiv 2 \imod{3}, \\
       0 &  n\equiv 0,1 \imod{3}. \\
     \end{array}
     \right.
\]

 \subsection{Eta-Quotients of Level 1024}
   \label{1024}
Taking $m=1, k=64$ in Theorem \ref{4mthm} and $m=8, s=5$ in Theorem \ref{psipsithm} yields
 \begin{equation}
   \label{di}
  \frac{B(1,0,256,q)-B(4,4,65,q)}{2} =q\psi(q^8)\phi(-q^{64}),
   \end{equation}
	and
\begin{equation}
   \label{di2}
  \frac{B(5,4,52,q)-B(13,4,20,q)}{2} =q^5\psi(-q^8)\psi(-q^{32}).
   \end{equation}

We find CL$(-1024) \cong C_8 $ and we list the reduced forms according to their genus
\begin{center}
\begin{tabular}{ | l | l | l | l | }
  \hline     
  \multicolumn{2}{|c|}{CL$(-1024) \cong C_8$} & $\left(\tf{-1}{p}\right)$ & $\left(\tf{2}{p}\right)$ \\
  \hline                   
  Principal Genus & (1,0,256), (4,4,65), (16,8,17), (16,-8,17) &+1 &+1 \\ \hline
  Second Genus & (5,4,52), (5,-4,52), (13,4,20), (13,-4,20) & +1 &$-1$ \\ \hline  
\end{tabular}.
\end{center}
In the above table, $p$ is taken to be odd so that it is coprime to $-1024=2^{10}$. Before moving on, we remark that the \textit{eta}-quotients $q\psi(q^8)\phi(-q^{64})$ and $q^5\psi(-q^8)\psi(-q^{32})$ are considered in \cite[p.232]{kohler}.
Using Table \ref{sunwilmc}, we find  
\begin{equation}
   \label{aone}
  A(q):=\frac{B(1,0,256,q)-B(4,4,65,q)+\sqrt{2}(B(5,4,52,q)-B(13,4,20,q))}{2},
   \end{equation}
to be multiplicative. The action of $T_p$ on $A(q)$ is given in Table \ref{tptab21024} for primes $p$ with $\left( \tf{-1024}{p}\right)=1$.
   \begin{table}[htb]  
	\caption{} \label{tptab21024}
     \begin{center} 
     \resizebox{.98\hsize}{!}{
       \begin{tabular}{|l|c|c|c|c|c|}              \hline 
                  & $(1,0,256,p)>0$ & $(4,4,65,p)>0$ & $(16,8,17,p)>0$ & $(5,4,52,p)>0$ & $(13,4,20,p)>0$\\ \hline
         $A(q) \xrightarrow{T_p}$   & $2A(q)$    & $-2A(q)$ & 0 &$\sqrt{2}A(q)$   & $-\sqrt{2}A(q)$      \\ \hline
       \end{tabular}
       }
			\begin{flushright}
			 .
			 \end{flushright}
       \label{}
     \end{center}
   \end{table}
	
	 If $\left(\tf{-1024}{p}\right)=-1$, then $A(q)$ is an eigenform for $T_p$ with eigenvalue 0. $A(q)$ is an eigenform for $T_2$  with eigenvalue $0$.

   Using \eqref{heckedef4} with our computed eigenvalues, we find 

\begin{equation}
   \label{272a1def}
[p^{\nu}]A(q) = \left\{ \begin{array}{ll}
 0 &  p=2,~~\nu>0, \\
       1+\nu &  (1,0,256,p)>0, \\
       (-1)^{\nu}(1+\nu) &  (4,4,65,p)>0, \\
       U(\nu) &  (5,4,52,p)>0, \\
       V(\nu) &  (13,4,20,p)>0, \\
       (-1)^{\tf{\nu}{2}}\frac{(-1)^{\nu}+1}{2} &  (16,8,17,p)>0, \\
       \frac{(-1)^{\nu}+1}{2} &  \left(\frac{-1024}{p}\right)=-1,
     \end{array}
     \right.
\end{equation}

with

\begin{equation}
   U(n):=\df{\sin(\pi(n+1)/4)}{\sin(\pi/4)},
   \end{equation}
   and
   \begin{equation}
   V(n):=\df{\sin(3\pi(n+1)/4)}{\sin(3\pi/4)}.
   \end{equation}

	\noindent
The functions $U(n)$ and $V(n)$ are periodic and we tabulate their values in Table \ref{272uv}.

      \begin{table}[htb]  
			\caption {} 	\label{272uv} 
     \begin{center} 
       \begin{tabular}{|l|c|c|}              \hline 
         $n$ & $U(n)$ & $V(n)$ \\ \hline
         $n \equiv 0 \pmod{8}$   & $1$    & $1$            \\ \hline
         $n \equiv 1 \pmod{8}$     & $\sqrt{2}$    & $-\sqrt{2}$           \\ \hline
         $n \equiv 2 \pmod{8}$      & $1$    & $1$     \\ \hline
         $n \equiv 3 \pmod{8}$      & $0$   & $0$     \\ \hline
         $n \equiv 4 \pmod{8}$      & $-1$   & $-1$     \\ \hline
         $n \equiv 5 \pmod{8}$      & $-\sqrt{2}$    & $\sqrt{2}$     \\ \hline
         $n \equiv 6 \pmod{8}$      & $-1$    & $-1$     \\ \hline
         $n \equiv 7 \pmod{8}$      & $0$    & $0$     \\ \hline
       \end{tabular}
			\begin{flushright}
			 .
			 \end{flushright}
     \end{center}
   \end{table}

		We mention that the Weber class polynomial for discriminant $-1024$ is
   \begin{align*}
  W_{-1024}(x):=& x^8 - 2363648x^7 - 14141504x^6 - 33443840x^5 - 9272384x^4\\
  &-6554624x^3 - 493568x^2 - 278528x - 128.
  \end{align*}
	The factorization pattern of $W_{-1024}(x) \imod{p}$ does not distinguish between primes represented by $(5,4,52)$ and $(13,4,20)$. As in previous examples, one can use the remainder criteria of \cite{cox} and \cite{voight} to distinguish between primes that are represented by different forms of discriminant $-1024$.\\

 Define $S_1$ to be the set of primes $p$ represented by $(1,0,256)$ or $(4,4,65)$. We also define $S_2,S_3$ to be the set of primes $p$ represented by $(5,4,52)$, $(13,4,20)$, respectively. Lastly, we let $S_4$ be the set of primes $p$ with $(16,8,17,p)>0$ or $\left(\tf{-1024}{p}\right)=-1$.

   We find that the formula for $[q^n]A(q)$ is given by
   
\begin{equation}
\label{272a1form}
(-1)^{t+\tf{s}{2}}\cdot\delta_{a,0}\prod_{p_1 \in S_1}(1+\operatorname{ord}_{p_1}(n))\prod_{p_2 \in S_2}U(\operatorname{ord}_{p_2}(n))\prod_{p_3 \in S_3}V(\operatorname{ord}_{p_3}(n))\prod_{p_4 \in S_4}\frac{(-1)^{\operatorname{ord}_{p_4}(n)}+1}{2},
\end{equation}
where $a=\operatorname{ord}_{2}(n)$, $t$ is the number of primes factors $p$ of $n$, counting multiplicity, that are represented by $(4,4,65)$, and $s$ is the number of primes factors of $n$, counting multiplicity, that are represented by $(16,8,17)$. \\
	
	 As mentioned earlier, we are able to employ congruences to extract the coefficients of $q\psi(q^8)\phi(-q^{64})$ and $q^5\psi(-q^8)\psi(-q^{32})$ . We obtain
   \[
[q^n]q\psi(q^8)\phi(-q^{64})=\left\{ \begin{array}{ll}
       [q^n]A(q)&  n\equiv 1\pmod{8}, \\
       0 &  \mbox{otherwise},\\
     \end{array}
     \right.
\]

and

 \[
[q^n]q^5\psi(-q^8)\psi(-q^{32})=\left\{ \begin{array}{ll}
       [q^n]A(q)/\sqrt{2} &  n\equiv 5\pmod{8}, \\
       0 &  \mbox{otherwise}.\\
     \end{array}
     \right.
\]

Given a positive integer $n$, we define $r_{i}$ to be the number of prime factors of $n$ with $p_2 \in S_2$ such that $\operatorname{ord}_{p_2}(n)\equiv i\pmod{8}$. Similarly, $s_{i}$ is the number of prime factors of $n$ with $p_3 \in S_3$ and $\operatorname{ord}_{p_3}(n)\equiv i\pmod{8}$.\\

\noindent
   With the notation of $r_i$ and $s_i$ we have
   \begin{equation}
   \prod_{p_2 \in S_2}U(\operatorname{ord}_{p_2}(n)) =\delta_{r_3+r_7,0} \cdot(-1)^{r_4+r_5+r_6}\cdot 2^{\tf{r_1+r_5}{2}},
   \end{equation}
   and
   \begin{equation}
   \prod_{p_3 \in S_3}V(\operatorname{ord}_{p_3}(n)) = \delta_{s_3+s_7,0} \cdot(-1)^{s_1+s_4+s_6}\cdot 2^{\tf{s_1+s_5}{2}},
   \end{equation}
   where $\delta_{i,j}$ is 1 if $i=j$ and 0 otherwise. Letting $k_i:=r_i+s_i$ for $i=1,2$, we obtain
   \begin{equation}
   \label{272calc1}
   \prod_{p_2 \in S_2}U(\operatorname{ord}_{p_2}(n))\prod_{p_3 \in S_3}V(\operatorname{ord}_{p_3}(n)) =\delta_{k_3+k_7,0}\cdot (-1)^{s_1+r_5+k_4+k_6}\cdot 2^{\tf{k_1+k_5}{2}}
   \end{equation}
   and
   \begin{equation}
   \label{272calc2}
   \prod_{p_2 \in S_2}V(\operatorname{ord}_{p_2}(n))\prod_{p_3 \in S_3}U(\operatorname{ord}_{p_3}(n)) =\delta_{k_3+k_7,0} \cdot(-1)^{r_1+s_5+k_4+k_6}\cdot 2^{\tf{k_1+k_5}{2}}.
   \end{equation}
   Using \eqref{272calc1} and \eqref{272calc2} we reformulate \eqref{272a1form} as
   \begin{equation}
\label{272alform2}
[q^n]A(q) = (-1)^{t+\tf{s}{2}+s_1+r_5+k_4+k_6}\cdot\delta_{a+k_3+k_7,0}\cdot 2^{\tf{k_1+k_5}{2}}\prod_{p_1 \in S_1}(1+\operatorname{ord}_{p_1}(n))\prod_{p_4 \in S_4}\frac{(-1)^{\operatorname{ord}_{p_4}(n)}+1}{2}.
\end{equation}
We obtain
\begin{equation}
[q^n]2q\psi(q^8)\phi(-q^{64})=[q^n]A(q)(1+(-1)^{k_1+k_5}),
   \end{equation}
      \begin{equation}
[q^n]2\sqrt{2}q^5\psi(-q^8)\psi(-q^{32})=[q^n]A(q)(1-(-1)^{k_1+k_5}),
   \end{equation}
thus the Fourier coefficients of $2q\psi(q^8)\phi(-q^{64})$ are given by
   \begin{equation}
(1+(-1)^{k_1+k_5})(-1)^{b+t+\tf{s}{2}+s_1+r_5+k_4+k_6}\cdot\delta_{a+k_3+k_7,0}\cdot 2^{\tf{k_1+k_5}{2}}\prod_{p_1 \in S_1}(1+\operatorname{ord}_{p_1}(n))\prod_{p_4 \in S_4}\frac{(-1)^{\operatorname{ord}_{p_4}(n)}+1}{2}.
   \end{equation}
and the Fourier coefficients of $2\sqrt{2}q^5\psi(-q^8)\psi(-q^{32})$ are given by
   \begin{equation}
(1-(-1)^{k_1+k_5})(-1)^{b+t+\tf{s}{2}+s_1+r_5+k_4+k_6}\cdot\delta_{a+k_3+k_7,0}\cdot 2^{\tf{k_1+k_5}{2}}\prod_{p_1 \in S_1}(1+\operatorname{ord}_{p_1}(n))\prod_{p_4 \in S_4}\frac{(-1)^{\operatorname{ord}_{p_4}(n)}+1}{2}.
   \end{equation}

   \subsection{An Eta-Product of Level 1872}
   \label{-1872}
	
	Applying Theorem \ref{6mthm} with $m=12$ and $s=7$, yields
   \begin{equation}
   \label{diff1872}
   \df{B(72,12,7,q) - B(72,60,19,q)}{2}=\df{B(7,2,67,q) - B(19,16,28,q)}{2}=q^{7}E(Q)E(Q^{13}),
   \end{equation}
   where we define $Q:=q^{12}$.\\
   We have CL$(-1872) \cong C_4 \times C_4$ and
\begin{center}
\begin{tabular}{ | l | l | l | l | l |}
  \hline     
  \multicolumn{2}{|c|}{CL$(-1872) \cong C_4 \times C_4$}& $\left(\tf{p}{3}\right)$ & $\left(\tf{p}{13}\right)$ & $\left(\tf{-1}{p}\right)$ \\
  \hline                   
  Principal Genus & $(1,0,468)$, $(4,0,117)$,$(9,0,52)$,$(13,0,36)$ &$+1$ &$+1$ &$+1$ \\ \hline
  Second Genus & $(7,2,67)$, $(7,-2,67)$, $(19,16,28)$, $(19,-16,28)$ &$+1$ &$-1$ &$-1$ \\ \hline
  Third Genus & $(8,4,59)$, $(8,-4,59)$, $(11,8,44)$, $(11,-8,44)$ &$-1$ &$-1$ &$-1$ \\ \hline 
  Fourth Genus & $(9,6,53)$, $(9,-6,53)$, $(17,10,29)$, $(17,-10,29)$ &$-1$ &$+1$ &$+1$ \\ \hline   
\end{tabular}
\begin{flushright}
			 .
			 \end{flushright}
   \end{center}
	In the above table, $p$ is taken to be coprime to $-1872$ and represented by the given genus.
   
   Using the bottom row of Table \ref{sunwilmc} with $A=(7,2,67), B=(11,8,44)$ we find
   \begin{equation}
   \label{m11872def}
   A(q):=\df{B(1,0,468,q)+B(13,0,36,q)-B(4,0,117,q)-B(9,0,52,q)}{2}+ 2q^{7}E(Q)E(Q^{13})
   \end{equation}
   is multiplicative.

    Appropriately applying \eqref{mod31}, \eqref{phieven}, and \eqref{mod32} to
   \begin{equation}
	\label{simp187233}
   \df{B(1,0,468,q)+B(13,0,36,q)-B(4,0,117,q)-B(9,0,52,q)}{2},
   \end{equation}
	we see a wonderful cancellation that transforms \eqref{simp187233} into
   \begin{equation}
   \label{simp21872}
   q\left[\phi(Q^{39})f(Q^2,Q^4) + Q\phi(Q^3)f(Q^{26},Q^{52}) -2Q^5\psi(Q^6)f(Q^{13},Q^{65})-2Q^{10}\psi(Q^{78})f(Q,Q^5)\right].
   \end{equation}
    Since $Q:=q^{12}$, the Fourier expansion of \eqref{simp21872} only contains terms with exponents congruent to $1\imod{12}$. Similarly, the Fourier expansion of $q^{7}E(Q)E(Q^{13})$ only contains terms with exponents congruent to $7\imod{12}$. Hence congruences can be employed to extract coefficients of \eqref{diff1872} from the completion \eqref{m11872def}.\\
    
    The expression \eqref{simp21872} is also discussed in \cite[p.181]{kohler} where the eta product $q^7E(Q)E(Q^{13})$ is related to the quadratic fields $\mathbb{Q}(\sqrt{-3}),\mathbb{Q}(\sqrt{-13})$, and $\mathbb{Q}(\sqrt{-39})$.\\
		
    In \cite{gord}, Gordon and Hughes consider the product $q^{7}E(Q)E(Q^{13})$, and introduce the function
    \begin{equation}
   \label{ghc}
   h(q)=q\left[\phi(-Q^{39})f(Q^2,Q^4) + Q\phi(-Q^3)f(Q^{26},Q^{52}) -Q^5\psi(Q^6)f(Q^{13},Q^{65})-Q^{10}\psi(Q^{78})f(Q,Q^5)\right].
   \end{equation}
   One can verify $h(q) + 2q^{7}E(Q)E(Q^{13})$ is an eigenform for some, but not all Hecke operators. Indeed, $h(q) + 2q^{7}E(Q)E(Q^{13})$ is not an eigenform for $T_7, T_{11}, T_{17}$, among others. We note the striking similarlity between \eqref{simp21872} and $h(q)$. We also comment that Gordon and Hughes were well aware that $h(q)$ was not an eigenform for all $T_p$, and did not make any erroneous claims regarding the action of $T_p$ on $h(q)$.\\

   The action of $T_p$ on the forms of discriminant $-1872$ are omitted since the computations are analogous to previous examples. The action of $T_p$ on $A(q)$ is given in Table \ref{tabb} for primes $p$ with $\left( \tf{p}{3}\right)=\left( \tf{-13}{p}\right)=1$. Note the property $\left( \tf{p}{3}\right)=\left( \tf{-13}{p}\right)=1$ is equivalent to the prime $p\neq13$ being represented by the principal genus or second genus, according to the table of genera at the beginning of this example.
	
   \begin{table}[htb]  
	\caption{} \label{tabb}
     \begin{center} 
     \resizebox{.98\hsize}{!}{
       \begin{tabular}{|l|c|c|c|c|c|c|}              \hline 
                  & $(1,0,468,p)>0$ & $(4,0,117,p)>0$ & $(13,0,36,p)>0$ & $(9,0,52,p)>0$ & $(7,2,67,p)>0$& $(19,16,28,p)>0$\\ \hline
         $A(q) \xrightarrow{T_p}$   & $2A(q)$    & $-2A(q)$ & $2A(q)$ & $-2A(q)$    & $2A(q)$    & $-2A(q)$    \\ \hline
       \end{tabular}
       }
			\begin{flushright}
			 .
			 \end{flushright}
     \end{center}
   \end{table}.
	
	For any prime $p\neq 13$ that does not have $\left( \tf{p}{3}\right)=\left( \tf{-13}{p}\right)=1$, we find $A(q)$ is an eigenform under $T_p$ with eigenvalue 0. $A(q)$ is an eigenform for $T_{13}$ with eigenvalue 1. Hence $A(q)$ is an eigenform for all Hecke operators. Employing \eqref{heckedef4} we obtain
	
   \begin{equation}
   \label{1872a1def}
[q^{p^{\nu}}]A(q) = \left\{ \begin{array}{ll}
       0 &  p=2,3,~~ \nu > 0,\\
       1 &  p=13, \\
       1+\nu &  p \neq 13 ,(1,0,468,p)+(13,0,36,p)+ (7,2,67,p)>0, \\
       (-1)^{\nu}(1+\nu) &  (4,0,117,p)+ (9,0,52,p)+ (19,16,28,p)>0, \\
       (-1)^{\frac{\nu}{2}}\frac{(-1)^{\nu}+1}{2} & -\left(\frac{p}{3}\right)=\left(\frac{-13}{p}\right)=1,\\
       \frac{(-1)^{\nu}+1}{2} &  \left(\frac{-1872}{p}\right)=-1,
     \end{array}
     \right.
\end{equation}

   Since $|$CL$(-1872)| = 16$, the associated Weber class polynomial, $W_{-1872}(x)$, is a degree 16 polynomial. Explicitly,
   \begin{align*}
  W_{-1872}(x):=& x^{16}-8x^{15}+24x^{14}-34x^{13}+x^{12}+246x^{11}-1094x^{10}+2574x^9-4200x^8+5608x^7\\
  &-5144x^6+858x^5+4189x^4-5166x^3+2814x^2-750x+69.
  \end{align*}
   
   Following \cite{cox} and \cite{voight}, one can use the remainder criteria that we have seen in previous examples to determine which primes are represented by a given form of discriminant $-1872$. The coefficients in this criteria become rather large and so we omit the computations. We remark that the procedure is completely analogous to previous examples.\\
		
		Let $S_1$ be the set of primes $p$ with $\left( \tf{p}{3}\right)=\left( \tf{-13}{p}\right)=1$. Employing \eqref{1872a1def} we obtain
   
\begin{equation}
\label{1872alform}
[q^n]A(q) = (-1)^{t_1+\tf{s}{2}}\cdot\delta_{\operatorname{ord}_{2}(n)+\operatorname{ord}_{3}(n),0}\prod_{p_1 \in S_1}(1+\operatorname{ord}_{p_1}(n))\prod_{p_2 \notin S_1 \cup \{2,3,13\}}\frac{(-1)^{\operatorname{ord}_{p_2}(n)}+1}{2},
\end{equation}
   where $n$ has $t_1$ prime factors with $(4,0,117,p)+ (9,0,52,p)+ (19,16,28,p)>0 $, $t_2$ prime factors with $(4,0,117,p)+ (9,0,52,p)+ (7,2,67,p)>0 $, $s$ prime factors with $-\left(\frac{p}{3}\right)=\left(\frac{-13}{p}\right)=1$, and the right most product is taken over primes $p_2 \notin S_1 \cup \{2,3,13\}$. All prime factors are counted with multiplicity, and similar to the previous example, $s$ odd implies \eqref{1872alform} vanish.\\
	
 As mentioned earlier, we can write the Fourier coefficients of $q^7E(Q)E(Q^{13})$ by employing congruences. We have
\[
[q^n]q^7E(Q)E(Q^{13})=\left\{ \begin{array}{ll}
       \tf{[q^n]A(q)}{2} &  n\equiv 7 \imod{12}, \\
       0 &  \mbox{otherwise}. \\
     \end{array}
     \right.
\]   
   
We arrive at the formula 
   \begin{equation}
\tf{(1-(-1)^{t_1+t_2})}{4}(-1)^{t_1+\tf{s}{2}}\cdot\delta_{\operatorname{ord}_{2}(n)+\operatorname{ord}_{3}(n),0}\prod_{p_1 \in S_1}(1+\operatorname{ord}_{p_1}(n))\prod_{p_2 \notin S_1 \cup \{2,3,13\}}\frac{(-1)^{\operatorname{ord}_{p_2}(n)}+1}{2}
   \end{equation}
for the Fourier coefficient $[q^n]q^7E(Q)E(Q^{13})$.

  \section{Concluding Remarks}
   \label{outlook}
	
	We note that the \textit{eta}-product of level 71, discussed in section \ref{examples71}, has many similarities with the \textit{eta}-product of level 47, discussed in section \ref{examples}, and also with the \textit{eta}-product of level 23 which has been previously discussed \cite{blij}. In the future we hope to continue the discussion of the \textit{eta}-products $q^{\tf{p+1}{24}}E(q)E(q^p)$ of level $p \equiv -1 \imod{24}$ where $p$ is prime.\\

	We lastly comment that the authors purposefully did not include any examples which are associated with a class group isomorphic to $C_4 \times C_2^r$, with $0\leq r \leq 4$. Much more can be said for these types of examples, and the authors hope to discuss these examples in a separate paper.

   \section{Acknowledgement}
   \label{ack}
   
   We are grateful to Hamza Yesilyurt, John Voight, Li-Chien Shen, and Kenneth Williams for helpful discussions, and to Keith Grizzell and Elizabeth Loew for a careful reading of the manuscript. We would like to thank Jean-Pierre Serre for bringing references \cite{blij}, \cite{fricke}, and \cite{shoe} to our attention. Lastly we would like to thank the anonymous referee for many valuable suggestions.

\end{document}